\documentclass[12pt]{article}

\usepackage[margin=1.5in]{geometry}
\usepackage{mathrsfs} 
\usepackage{amsmath}
\usepackage{amssymb}
\usepackage{amsthm}
\usepackage{amsfonts}
\usepackage{multicol}
\usepackage[numbers,sort&compress]{natbib}
\usepackage{subfigure}
\usepackage{enumitem}
\setlist[enumerate]{label={\upshape(\alph*)}}

\usepackage{color}

\usepackage{tikz}
\usetikzlibrary{calc}
\tikzstyle{vertex}=[circle, draw, inner sep=0pt, minimum size=4pt,fill=black]

\tikzstyle{hollowvertex}=[circle, draw, inner sep=0pt, minimum size=4pt]

\tikzstyle{namedvertex}=[circle, draw, inner sep=1pt, minimum size=12pt]

\tikzstyle{phantomvertex}=[circle, draw, inner sep=0pt, minimum size=4pt,color=white]

\usetikzlibrary{decorations.pathreplacing}

\newtheorem{defn}{Definition}[section]
\newtheorem{lemma}[defn]{Lemma}
\newtheorem{ques}[defn]{Question}
\newtheorem{prop}[defn]{Proposition}
\newtheorem{theorem}[defn]{Theorem}
\newtheorem{corollary}[defn]{Corollary}

\newtheorem{conj}[defn]{Conjecture}

\begin{document}

\title{On the largest real root of the independence polynomial of a unicyclic graph}
\author{Iain Beaton\\
\small Dept. of Mathematics \& Statistics\\[-0.8ex]
\small Acadia University\\[-0.8ex] 
\small Wolfville, NS\\
\small\tt iain.beaton@acadiau.ca\\
\and
Ben Cameron\thanks{Corresponding author} \thanks{Research support from the Natural Sciences and Engineering Research Council of Canada (NSERC) is gratefully acknowledged by Ben Cameron (grants DGECR-2022-00446 and RGPIN-2022-03697)}\\
\small Dept. of Computing Science\\[-0.8ex]
\small The King's University\\[-0.8ex] 
\small Edmonton, AB\\
\small\tt ben.cameron@kingsu.ca\\
\\
}
\date{\today}

\maketitle

%
%
%

\tikzset{bignode/.style={minimum size=3em,}}

\begin{abstract}
The independence polynomial of a graph $G$, denoted $I(G,x)$, is the generating polynomial for the number of independent sets of each size.  The roots of $I(G,x)$ are called the \textit{independence roots} of $G$. It is known that for every graph $G$, the independence root of smallest modulus, denoted $\xi(G)$, is real. The relation $\preceq$ on the set of all graphs is defined as follows, $H\preceq G$ if and only if $I(H,x)\ge I(G,x)\text{ for all }x\in [\xi(G),0].$ 

We find the maximum and minimum connected unicyclic and connected well-covered unicyclic graphs of a given order with respect to $\preceq$. This extends 2013 work by Csikv\'{a}ri where the maximum and minimum trees of a given order were determined and also answers an open question posed in the same work. Corollaries of our results give the graphs that minimize and maximize $\xi(G)$ among all connected (well-covered) unicyclic graphs. We also answer more related open questions posed by Oboudi in 2018 and disprove a conjecture due to Levit and Mandrescu from 2008.

\noindent
{\bf Keywords:} independence polynomial; independence root; smallest modulus; independent set
\end{abstract}
\maketitle

\section{Introduction}\label{sec:intro}
All graphs considered in this paper are finite, simple, and undirected. For a vertex $v$ of a graph $G$, the \textit{closed neighbourhood} of $v$, is the set $N[v]=\{u\in V(G): u\sim v\}\cup\{v\}$ and for an edge $e=uv$, the closed neighbourhood is the set $N[e]=N[u]\cup N[v]$. For $S\subseteq V(G)$, $G[S]$ denotes the subgraph induced by $S$. For $S\subseteq V(G)$, let $G-S$ denote the graph $G[V(G)-S]$. If $S=\{v\}$, then we simply write $G-v$ instead of $G-\{v\}$. For an edge $e\in E(G)$, the graph $G-e$ is the graph obtained from $G$ by removing the edge $e$. For a pair of nonadjacent vertices $u$ and $v$ in $V(G)$, the graph $G+uv$ is the graph obtained from $G$ by adding the edge $uv$. The \textit{disjoint union} of two graphs $G$ and $H$, denotes $G\cup H$ is the graph $(V(G)\cup V(H),E(G)\cup E(H))$, and $kG$ denoted the disjoint union of $k$ copies of $G$. A subset of vertices of a graph $G$ is called \textit{independent} if the all the vertices in the subset are pairwise nonadjacent.  The \textit{independence number} of  $G$ is the size of the largest independent set in $G$ and is denoted by $\alpha(G)$. The \textit{independence polynomial} of $G$, denoted by $I(G,x)$ is defined by 
\[ I(G,x)=\sum_{k=0}^{\alpha}i_kx^k,\] 
where $i_k$ is the number of independent sets of size $k$ in $G$. The roots of $I(G,x)$ are called the \textit{independence roots} of $G$ and they, along with other properties of the independence polynomial, have been subject to much investigation since the introduction of $I(G,x)$ in 1983 by Gutman and Harary \cite{INDFIRST}. The independence polynomial generalizes the matching-generating polynomial, but where the latter has all real roots \cite{HEILMANN}, the former has roots dense in $\mathbb{C}$ \cite{INDROOTS}.  However, when restricting $G$ to the class of claw-free graphs, Chudnovsky and Seymour \cite{Chudnovsky2007} showed that $G$ has all real independence roots. Moreover, for every graph, the independence root of smallest modulus is always real and in the interval $[-1,0)$ \cite{Csikvari2013}. For a graph $G$ let $\xi(G)$ denote the independence root of smallest modulus. Note that since $\xi(G)$ is always a negative real number, it is equivalently referred to as the largest real independence root. 

One question of interest is: given a family of graphs, what is the maximum/minimum modulus  of an independence root of a graph in that family?  Linear bounds on the maximum modulus of independence roots of well-covered graphs of fixed order \cite{BDN2000} and respective exponential bounds are known for all graphs and forests of a fixed order \cite{BrownCameron2020}. Although the bounds in all cases are tight asymptotically, it remains unknown which graph (or family of graphs) exactly achieves the maximum values. Only recently have bounds on the independence root of smallest modulus in a family of graphs been considered, and in these cases the graphs with this minimum value can be found. For example, if $T$ is a tree of order $n$, then Csikv\'{a}ri~\cite{Csikvari2013posetII} showed, using their remarkable generalized tree shift~\cite{Csikvari2010}, that $\xi(P_n)\le \xi(T)\le \xi(S_n)$ where $P_n$ and $S_n$ are the path and star of order $n$, respectively. It follows from this and a result in \cite{Oboudi2018treeroots} that among all connected graphs $G$ of order $n$, $\xi(P_n)\le \xi(G)\le \xi(K_n)$, where $K_n$ is the complete graph of order $n$. However, the problem is open for all other families of graphs. 
 
To find graphs with maximum and minimum values of $\xi$ in a given family, Csikv\'{a}ri~\cite{Csikvari2013posetII} introduced (see also \cite{cliquepolys}) the relation $\preceq$ on the set of all graphs, defined as follows, 

$$H\preceq G\text{ if and only if }I(H,x)\ge I(G,x)\text{ for all }x\in [\xi(G),0].$$ 
Since for every $x\in (\xi(G),0]$, $I(G,x)>0$, it follows that $H\preceq G$ implies $\xi(H)\le \xi(G)$. This stronger relation affords some advantages so we will work with this instead of directly trying to order graphs by their largest real root. One of these advantages is the ability to track graphs under $\preceq$ subject to small local changes, like deleting and adding edges.  We also note that the relation $\preceq$ can be made strict in the following sense, $H\prec G$ if and only if $H\preceq G$ and $I(H,x)\neq I(G,x)$. This can require knowing when two graphs share an independence polynomial which is another recent area of interest on the independence polynomial.

We say that two graphs $G$ and $H$ are \textit{independence equivalent}, denoted $G \sim H$, if they have the same independence polynomial. For example, $C_n\sim D_n$ \cite{Oboudi2018} where $D_n$ is the graph in Figure~\ref{fig:Ln}. Independence equivalence is clearly an equivalence relation on the set of all unlabelled graphs, so we define the \textit{independence equivalence class} of a graph $G$, denoted $[G]$, to be the set of all graphs that are independence equivalent to $G$. A graph $G$ is called \textit{independence unique} if $[G]=\{G\}$. Interest in independence equivalence problems has been growing recently. Oboudi \cite{Oboudi2018} classified all connected graphs $G$ such that $G\sim C_n$. Li et al. \cite{Li2016} showed $P_n$ is independence unique among all connected graphs. With Brown, the authors extended this work by considering graphs in $[C_n]$ and $[P_n]$ that are disconnected~\cite{BeatonBrownCameron2019}. Very recently, Ng~\cite{Ng2021cycles,Ng2022paths} completely determined $[C_n]$ and $[P_n]$ for all $n$.  Independence equivalence was also vital for Barik et al.'s work determining all graphs whose independence fractals are line segments \cite{Barik2020}.

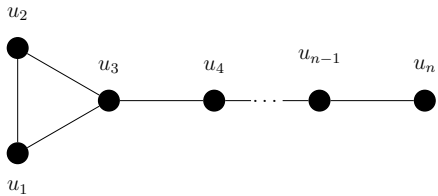
\begin{figure}[htp]
\def\c{0.7}
\def\r{2}
\centering
\scalebox{\c}{
\begin{tikzpicture}
\begin{scope}[every node/.style={circle,thick,fill,draw}]
    \node[label=above:$u_2$,draw] (2) at (1.13397*\r,0.5*\r) {};
    \node[label=above:$u_3$,draw] (3) at (2*\r,0*\r) {};
    \node[label=above:$u_{n-1}$,draw] (4) at (4*\r,0*\r) {};
    \node[label=above:$u_n$,draw] (5) at (5*\r,0*\r) {};
    \node[label=below:$u_1$,draw] (7) at (1.13397*\r,-0.5*\r) {};
    \node[label=above:$u_4$,draw] (8) at (3*\r,0*\r) {};   
\end{scope}

\begin{scope}

    \path [-] (2) edge node {} (3);
    \path [-] (4) edge node {} (5);
    \path [-] (3) edge node {} (7);

    \path [-] (2) edge node {} (7);
    \path [-] (3) edge node {} (8);
    
\end{scope}

\path (8) -- node[auto=false]{\ldots} (4);
\path [-] (8) edge node {} (3*\r+.7,0*\r) ;
\path [-] (4*\r-.7,0*\r) edge node {} (4);
\end{tikzpicture}}
\caption{The graph $D_n$}%
\label{fig:Ln}%
\end{figure}

Let $\mathcal{A}$ be the set of all distinct independence equivalence classes of all graphs. Although not all elements of $\mathcal{A}$ are related via $\preceq$, (see for example $[\overline{K_3}]$ and $[K_2]$), we still have the following. 

\begin{theorem}[\cite{Csikvari2013posetII}]\label{thm:partialorder}
$(\mathcal{A},\preceq)$ is a partial order.
\end{theorem}

This paper is structured as follows. We find the maximum and minimum graphs with respect to $\preceq$ among all connected unicyclic graphs of a given order in Section~\ref{sec:unicyclicgraphs} which answers an open question posed in \cite{Oboudi2018treeroots}. Then, in Section~\ref{sec:wcunicyclic}, we use characterizations from \cite{nowawcgirth} and \cite{ToppVolkmann1990} for well-covered trees and well-covered unicyclic graphs, respectively, to find the maximum and minimum graphs for both of these well-covered families with respect to $\preceq$. In Section~\ref{sec:counterexamples}, we provide an infinite family of counterexamples to a conjecture due to Levit and Mandrescu \cite{INDPOLY,LevitMandrescu2003UnimodalitySomeTrees}. These counterexamples also serve as answers two open questions in \cite{Oboudi2018treeroots}. We then disprove a conjecture of Oboudi \cite{Oboudi2018treeroots} by showing that $\preceq$ is not a total order on the set of independence equivalence classes of all trees of fixed order. We conclude with some open problems.

\section{Unicyclic graphs}\label{sec:unicyclicgraphs}
In this section we focus our attention on connected unicyclic graphs, answering the following question.  We note that being able to generate all connected unicyclic graphs of small order using \texttt{nauty}~\cite{nauty} was invaluable for building intuition for our results in this section and the rest of the paper.

\begin{ques}[\cite{Oboudi2018treeroots}]
What are the maximal and minimal elements with respect to $\preceq$ among the family of connected unicyclic graphs?
\end{ques}

Before we can answer this question though, we will need a few more preliminary results.

\begin{theorem}[\cite{Csikvari2013posetII}]\label{thm:subgraphlessthangraph}
If $G$ is a graph with $H$ a proper subgraph of $G$, then $H\prec G$.
\end{theorem}

\begin{theorem}[\cite{Oboudi2018treeroots}]\label{thm:recursivepartialorder}
Let $G$ and $H$ be graphs with $u\in V(G)$, $e\in E(G)$, $v\in V(H)$, and $f\in E(H)$. Then the following hold:
\begin{itemize}
\item[$(i)$] If $H-v\preceq G-u$ and $G-N[u]\preceq H-N[v]$, then $H\preceq G$. 
\item[$(ii)$] If $H-f\preceq G-e$ and $G-N[e]\preceq H-N[f]$, then $H\preceq G$. 
\end{itemize}

Moreover, if either of the relations are strict in the hypothesis of $(i)$ (resp. $(ii)$), then $H\prec G$.
\end{theorem}

Let $G$ be a graph with $\Delta(G)\ge 3$, $\delta(G)=1$, and let $u\in V(G)$ be a leaf. Let $v$ be a vertex in $G$ with $\deg(v)\ge 3$ such that it has the shortest distance to $u$ among all vertices with degree at least $3$. Let $w\neq u$ be a neighbour of $v$ not on the path from $u$ to $v$. Note that such a neighbour always exists as $\deg(v)\ge 3$. Define $G^{\star}_{u,v,w}=(G-vw)+uw$. Note that $G$ is connected if and only if $G^{\star}_{u,v,w}$ is connected. Furthermore it can easily be seen that $G^{\star}_{u,v,w}$ has one fewer leaf than $G$. This graph operation is illustrated in Figure \ref{fig:G*}.

\setcounter{subfigure}{0}
\begin{figure}[!h]
\def\c{0.5}
\def\r{1}
\centering
\subfigure[$G$]{
\scalebox{\c}{
\begin{tikzpicture}
\draw (0,0) ellipse (1.2 and 2.5);

\node[shape=circle,draw=black,fill=black, label=left:\Large $u$] (1) at (-6,0) {};
\node[shape=circle,draw=black,fill=black] (2) at (-4,0) {};
\node[shape=circle,draw=black,fill=black, label=above:\Large $v$] (3) at (-2,0) {};
\node[shape=circle,draw=black,fill=black, label=above:\Large $w$] (4) at (-0.25,1.5) {};
\node[shape=circle,draw=black,fill=black] (5) at (-0.25,0.5) {};
\node[shape=circle,draw=black,fill=black] (6) at (-0.25,-1.5) {};

\begin{scope}
    \path [-] (1) edge node {} (2);
    \path [-] (2) edge node {} (-3.5,0);
    \path [-] (-2.5,0) edge node {} (3);
    \path [-] (3) edge node {} (4);
    \path [-] (3) edge node {} (5);
    \path [-] (3) edge node {} (6);
\end{scope}

\begin{scope}
    \path (-3.5,0) -- node[auto=false]{\Large \ldots} (-2.5,0);
    \path (-0.25,0.5) -- node[auto=false]{\Large \vdots} (-0.25,-1.5);
\end{scope}
\end{tikzpicture}}}
\qquad
\subfigure[$G^{\star}_{u,v,w}$]{
\scalebox{\c}{
\begin{tikzpicture}
\draw (0,0) ellipse (1.2 and 2.5);

\node[shape=circle,draw=black,fill=black, label=left:\Large $u$] (1) at (-6,0) {};
\node[shape=circle,draw=black,fill=black] (2) at (-4,0) {};
\node[shape=circle,draw=black,fill=black, label=above:\Large $v$] (3) at (-2,0) {};
\node[shape=circle,draw=black,fill=black, label=above:\Large $w$] (4) at (-0.25,1.5) {};
\node[shape=circle,draw=black,fill=black] (5) at (-0.25,0.5) {};
\node[shape=circle,draw=black,fill=black] (6) at (-0.25,-1.5) {};

\begin{scope}
    \path [-] (1) edge node {} (2);
    \path [-] (2) edge node {} (-3.5,0);
    \path [-] (-2.5,0) edge node {} (3);
    \path [-] (1) edge node {} (4);
    \path [-] (3) edge node {} (5);
    \path [-] (3) edge node {} (6);
\end{scope}

\begin{scope}
    \path (-3.5,0) -- node[auto=false]{\Large \ldots} (-2.5,0);
    \path (-0.25,0.5) -- node[auto=false]{\Large \vdots} (-0.25,-1.5);
\end{scope}
\end{tikzpicture}}}
\caption{A graph operation for graphs with $\Delta(G)\geq 3$ and $\delta=1$.}%
\label{fig:G*}%
\end{figure}
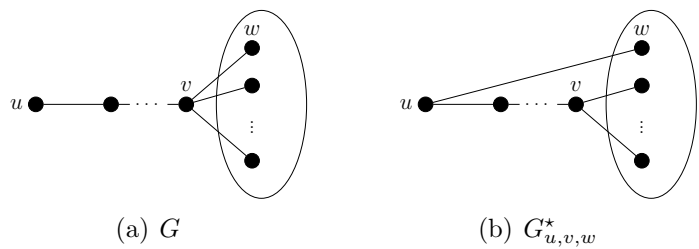

\begin{theorem}[\cite{Oboudi2018treeroots}]\label{thm:boldstaroperation}
If $G$ is a graph with $u,v,w$ as above, then  $G^{\star}_{u,v,w}\preceq G$.
\end{theorem}

\begin{theorem}[\cite{Oboudi2018}]\label{thm:cycleequivclass}
If $G$ is a connected unicyclic graph with $I(G,x)=I(C_n,x)$, then $G=C_n$ or $G=D_n$.
\end{theorem}

Let $U_n$ be the graph on $n$ vertices obtained by attaching $n-3$ vertices with an edge each to one vertex of a triangle, see Figure~\ref{fig:Un}. Note that $I(U_n,x)=(1+2x)(1+x)^{n-3}+x$ as every independent set is either the universal vertex or a subset of the $n-3$ leaves with at most one of the other two vertices in the triangle.

 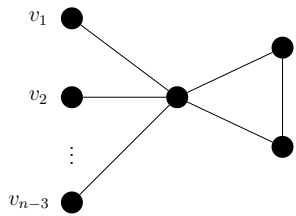
\begin{figure}[htb]
\def\c{0.7}
\def\r{2}
\centering
\scalebox{\c}{
\begin{tikzpicture}
\begin{scope}[every node/.style={circle,thick,draw,fill}]

     \node[label=left:$v_{n-3}$] (13) at (-1*\r,-1*\r) {}; 
     \node[label=left:$v_2$](11) at (-1*\r,0*\r) {};   
    \node[label=left:$v_1$](6) at (-1*\r,0.75*\r) {};
    
    \node(5) at (0*\r,0*\r) {};
\node(4) at (1*\r,sqrt{2/3}*\r) {};
    \node (3) at (1*\r,-sqrt{2/3}*\r) {}; 
      
\end{scope}

\begin{scope}
   
    \path [-] (5) edge node {} (4);
    
    \path [-] (4) edge node {} (3);
    \path [-] (5) edge node {} (3);
    
    \path [-] (6) edge node {} (5);
    \path [-] (5) edge node {} (11);
    \path [-] (5) edge node {} (13);
\end{scope}

\path (11) -- node[auto=false]{\vdots} (13);

\end{tikzpicture}}
\caption{The graph $U_n$}%
\label{fig:Un}%
\end{figure}

\begin{lemma}\label{lem:Tnindepunique}
If $G$ is a graph such that $I(G,x)=I(U_n,x)$, then $G\cong U_n$ or $n=4$ and $G\cong C_4$.
\end{lemma}
\begin{proof}
When $n=4$, we are done as $U_4=D_4$, so $[U_4]=\{C_4,U_4\}$ from Theorem~\ref{thm:cycleequivclass}. So suppose $n\ge 5$. We know that $G$ and $U_n$ must have the same number of vertices and edges as every subset of order $1$ is independent and the only subsets of order $2$ which are not independent are edges. Furthermore, since $I(U_n,x)=I(G,x)=(1+2x)(1+x)^{n-3}+x$, $G$ has independence number $n-2$, and $2$ maximum independent sets. Let $A$ and $B$ be the two independent sets of order $n-2$.

\noindent\textbf{Case 1:} $|A\cap B|=n-3$. 

In this case, let $b$ be the unique vertex in $B\setminus A$, $a$ be the unique vertex in $A\setminus B$, and $v$ be the unique vertex in $V(G)\setminus (A\cup B)$. Since $\alpha(G)=n-2$, it follows that $a\sim b$ and the subgraph induced by $A\cup B$ has exactly one edge. Therefore, every other edge in the graph must be incident with $v$. There are $n-1$ options for these edges and there are $n-1$ edges to account for, so $v$ is adjacent to every vertex in $A\cup B$. Therefore $G\cong U_n$.

\noindent\textbf{Case 2:} $|A\cap B|=n-4$. 

In this case, $B\setminus A$ and $A\setminus B$ each contain exactly two vertices and $A\cup B =V(G)$. Since $A$ and $B$ are independent sets, the only edges in the graph are
between vertices in $A\setminus B$ and vertices in $B\setminus A$. Therefore $G$ has at most $4$
edges which is a contradiction as G has $n\ge 5$ edges.
 
\end{proof}

We are now ready to provide the maximum and minimum connected unicyclic graphs with respect to $\preceq$.
\begin{theorem}\label{thm:maxminunicyclic}
If $G$ is a connected unicyclic graph on $n$ vertices such that $G\not\cong C_n$, $G\not\cong D_n$, and $G\not\cong U_n$, then $C_n\prec G\prec U_n.$
\end{theorem}
\begin{proof}

Let $G$ be a connected unicyclic graph of order $n$. We first show that $C_n\preceq G$. We proceed by induction on the number of leaves in $G$. If $G$ has no leaves, then $G\cong C_n$ so clearly $C_n\preceq G$. Now suppose $C_n\preceq G$ for all connected unicyclic graphs with at most $k\ge 0$ leaves and let $G$ be a connected unicyclic graph with $k+1$ leaves. Since $G$ has at least one leaf, $G$ also contains at least one vertex of degree $3$ or more. Let $u\in V(G)$ be a leaf and let $x\in V(G)$ be the vertex on the cycle of $G$ that is the shortest distance from $u$ (note $\deg(x)\ge 3$). Let $T$ be the tree of $G$ induced by all vertices such that there exists a path from $u$ to using either no vertices from the cycle or only $x$ from the cycle. Let $v\in V(G)$ have degree at least $3$ and be of shortest distance to $u$ among all such vertices in $G$. Note that $v\in V(T)$.

\noindent\textbf{Case 1:} $v=x$.

In this case, let $w$ be a neighbour of $x$ on the cycle. Then, since $G-vw$ is a tree, the graph $G-vw+uw$ is a connected unicyclic graph with one less leaf than $G$.  Thus, by Theorem~\ref{thm:boldstaroperation} and the inductive hypothesis, we have $C_n\preceq G^{\star}_{u,v,w}\preceq G$.

\noindent\textbf{Case 2:} $v\neq x$.

Let $w$ be a neighbour of $v$ with degree at least $2$ that is not on the shortest path between $u$ and $v$. Note that such a $w$ exists and is in $V(T)$, otherwise there is no path in $T$ between $v$ and $x$ which contradicts $G$ being connected. Now the graph $T+uw$ has a unique cycle that includes both $uw$ and $vw$ by our choice of $w$, so the graph $T-vw+uw$ is a tree and since $w$ had degree at least $2$, $T-vw+uw$ has one less leaf than $T$ and therefore $G^{\star}_{u,v,w}$ is a connected unicyclic graph with one less leaf than $G$. Therefore, by Theorem~\ref{thm:boldstaroperation} and the inductive hypothesis we have $C_n\preceq G^{\star}_{u,v,w}\preceq G$.

We now show that $G\preceq U_n$ by induction on $n$.  When $n=3$, the only unicyclic graph is $C_n=U_n$ so there is nothing to show. Suppose $H\preceq U_k$ for all connected unicyclic graphs $H$ of order $k$ and all $3\le k<n$. Now consider the connected unicyclic graph $G$ of order $n$. From above, we may assume $G\not\cong C_n$, so $G$ has at least one leaf. Let $v$ be a leaf of $G$ and $u$ be a leaf of $U_n$. Now, $G-v$ is a connected unicyclic graph of order $n-1$ and $U_n-u=U_{n-1}$ so $G-v\preceq U_n-u$ by the inductive hypothesis. We also have that $G-N[v]$ is a graph of order $n-2$ and $U_n-N[u]=\overline{K_{n-4}}\cup K_2$. It is not difficult to see that for any connected unicyclic graph of order $n$, its maximum degree is at most $n-1$, with equality achieved only for $U_n$. Therefore, there are at most $n-1$ edges incident with vertices in $N[v]$, so $G$ has at least one edge. Hence, $U_n-N[u]$ is a subgraph of $G-N[v]$ and by Theorem~\ref{thm:subgraphlessthangraph}, $U_n-N[u]\preceq G-N[v]$. Thus, by Theorem~\ref{thm:recursivepartialorder}, $G\preceq U_n$.

Therefore, $C_n\preceq G\preceq U_n$ for every connected unicyclic graph $G$. Finally, by Lemma~\ref{lem:Tnindepunique} and Theorem~\ref{thm:cycleequivclass}, we have $C_n\prec G\prec U_n$.
\end{proof}

\begin{corollary}
If $G$ is a connected unicyclic graph of order $n$, then $$\xi(C_n)\le \xi(G)\le \xi(U_n).$$
\end{corollary}

\section{Well-covered trees and unicyclic graphs}\label{sec:wcunicyclic}

A graph on $n$ vertices is called \textit{well-covered} if all of its maximal independent sets have the same size. It is called \textit{very well-covered} if all maximal independent sets are of size $n/2$. One construction that always produces a very well-covered graph is the graph star operation. Let $G$ be any graph. Form $G^\ast$, the {\it graph star} of $G$ (sometimes also called the {\it corona} of $G$ with $K_{1}$) from $G$ by attaching, for each vertex $v$ of $G$ a new vertex $v^\ast$ to $v$ with an edge (an edge incident with a leaf is called a \textit{pendant edge}), see Figure~\ref{fig:Pnstar}. It is not difficult to see that the graph $G^{\ast}$ is very well-covered for any graph $G$. As might be expected, independence polynomials of (very) well-covered graphs have been of particular interest. Given the restricted structure of their independent sets, well-covered graphs can have independence polynomials that behave quite differently from those of general graphs. For example, among all graphs of order $n$, the maximum modulus of an independence root is $O(3^{\frac{n}{3}})$ \cite{BrownCameron2020}, but among all well-covered graphs of order $n$, all independence roots are contained in the disk $|z|< n$ \cite{BDN2000}. 

\begin{figure}[htp]
\def\c{0.7}
\def\r{2}
\centering
\scalebox{\c}{
\begin{tikzpicture}
\begin{scope}[every node/.style={circle,thick,fill,draw}]
    \node[label=above:$u_1$,draw] (1) at (0*\r,0*\r) {};
    \node[label=above:$u_2$,draw] (2) at (1*\r,0*\r) {};
    \node[label=above:$u_3$,draw] (3) at (2*\r,0*\r) {};
    \node[label=above:$u_{n-1}$,draw] (4) at (3*\r,0*\r) {};
    \node[label=above:$u_n$,draw] (5) at (4*\r,0*\r) {};
    \node[label=below:$u_1^{\ast}$,draw] (6) at (0*\r,-1*\r) {};
    \node[label=below:$u_2^{\ast}$,draw] (7) at (1*\r,-1*\r) {};
    \node[label=below:$u_3^{\ast}$,draw] (8) at (2*\r,-1*\r) {};
    \node[label=below:$u_{n-1}^{\ast}$,draw] (9) at (3*\r,-1*\r) {};
    \node[label=below:$u_n^{\ast}$,draw] (10) at (4*\r,-1*\r) {};    
\end{scope}

\begin{scope}
    \path [-] (1) edge node {} (2);
    \path [-] (2) edge node {} (3);
    \path [-] (4) edge node {} (5);
    
    \path [-] (1) edge node {} (6);
    \path [-] (2) edge node {} (7);
    \path [-] (3) edge node {} (8);
    \path [-] (4) edge node {} (9);
    \path [-] (5) edge node {} (10);
    
\end{scope}

\path (3) -- node[auto=false]{\ldots} (4);
\path [-] (3) edge node {} (2*\r+.7,0*\r) ;
\path [-] (3*\r-.7,0*\r) edge node {} (4);
\end{tikzpicture}}
\caption{$P_n^{\ast}$}%
\label{fig:Pnstar}%
\end{figure}

\subsection{Well-covered trees}

Well-covered trees have a particularly nice characterization in terms of their leaves that we will make use of.

\begin{theorem}[\cite{nowawcgirth}]\label{thm:wctg6}
Let $G$ be a connected graph with girth at least $6$ which is isomorphic to neither $C_7$ nor $K_1$. Then $G$ is well-covered if and only if its pendent edges form a perfect matching.
\end{theorem}
\begin{corollary}\label{cor:wctreesaregraphstars}
A tree $T$ with at least $2$ vertices is well-covered if and only if $T=T'^{\ast}$ for some tree $T'$.
\end{corollary}

The independence polynomial of $G^{\ast}$ can be expressed in terms of the independence polynomial of $G$ by the following result.

\begin{theorem}[\cite{Levit2008}]\label{thm:stargraphformula}
If $G$ is a graph of order $n$, then $$I(G^{\ast},x)=(1+x)^nI\left( G,\frac{x}{1+x}\right).$$
\end{theorem}

We now have a helpful reduction when trying to determine when two graph stars are independence equivalent.

\begin{prop}[\cite{Levit2008}]\label{prop:graphstarsequal}
Let $G$ and $H$ be graphs. Then $I(G,x)=I(H,x)$ if and only if $I(G^{\ast},x)=I(H^{\ast},x)$.
\end{prop}

The real power of Theorem~\ref{thm:stargraphformula}, however, is the relationship it describes between the independence roots of $G^{\ast}$ and $G$.

\begin{lemma}\label{lem:relationresptostar}
Let $G$ and $H$ be graphs of order $n$. Then $H\preceq G$ (resp. $H\prec G$) if and only if $H^{\ast}\preceq G^{\ast}$ (resp. $H^{\ast}\prec G^{\ast}$).
\end{lemma}
\begin{proof}
Note that from Theorem~\ref{thm:stargraphformula} we know that the roots of $I(G^{\ast},x)$ are $\tfrac{r}{1-r}$ for all roots $r$ of $I(G,x)$ along with $-1$ with multiplicity $n-\alpha(G)$. Note that if $r_1< r_2 <0$, then $\tfrac{r_1}{1-r_1}<\tfrac{r_2}{1-r_2}$. Therefore as $\xi(G)\geq-1$ it follows that $\xi(G^{\ast})=\tfrac{\xi(G)}{1-\xi(G)}\geq -\tfrac{1}{2}$. Furthermore $x\in [\xi(G^{\ast}),0]$ if and only if $\tfrac{x}{1+x}\in [\xi(G),0]$. Since $\xi(G^{\ast})\geq -\tfrac{1}{2}$, it follows that $(1+x)^n>0$ for all $x\in [\xi(G^{\ast}),0]$. Therefore, $I(H,\tfrac{x}{1+x})\ge I(G,\tfrac{x}{1+x})$ for all $\tfrac{x}{1+x}\in [\xi(G),0]$ if and only if $(1+x)^nI(H,\tfrac{x}{1+x})\ge (1+x)^nI(G,\tfrac{x}{1+x})$ for all $x\in [\xi(G^{\ast}),0]$. Since $I(G^{\ast},x)=(1+x)^nI\left( G,\frac{x}{1+x}\right)$ and $I(H^\ast,x)=(1+x)^nI\left( H,\frac{x}{1+x}\right)$ by Theorem~\ref{thm:stargraphformula}, it follows that $H\preceq G$ if and only if $H^{\ast}\preceq G^{\ast}$. From this and Proposition~\ref{prop:graphstarsequal} it follows that $H\prec G$ if and only if $H^{\ast}\prec G^{\ast}$.
\end{proof}

From Lemma~\ref{lem:relationresptostar} and Theorem \ref{thm:wctg6}, the function $f(T)=T^{\ast}$ between the set of all trees of order $n$ and the set of all well-covered trees of order $2n$ is a bijection that preserves $\preceq$. 

\begin{theorem}[\cite{Oboudi2018treeroots}]\label{thm:treebounds}
If $G$ is a tree of order $n$ not equal to $P_n$ or $S_n$, then $P_n\prec T\prec S_n$.
\end{theorem}

\begin{corollary}\label{cor:wctreebounds}
Let $T$ be a well-covered tree of order $2n$ such that $T\neq P_{n}^{\ast}$ and $T\neq S_{n}^{\ast}$. Then
$P_{n}^{\ast}\prec T\prec S_{n}^{\ast}.$
\end{corollary}
\begin{corollary}
If $G$ is a well-covered tree of order $2n$, then 
$$\xi(P_n^{\ast})\le\xi(G)\le\xi(S_n^{\ast}).$$
\end{corollary}

\subsection{Well-covered unicyclic graphs}

We turn our attention now to well-covered unicyclic graphs. Let $\mathcal{KU}$ be the family of all graphs of the form $G^{\ast}$ where $G$ is a connected unicyclic graph. Unlike well-covered trees, there are well-covered unicyclic graphs that do not belong to $\mathcal{KU}$. In particular, there are well-covered unicyclic graphs of odd order, so finding the maximum and minimum well-covered unicyclic graphs of order $n$ will require more sophisticated techniques than simply applying Lemma~\ref{lem:relationresptostar}, although this lemma will be useful. We need to define a few more families of graphs before we can state Topp and Volkmann's \cite{ToppVolkmann1990} characterization of connected well-covered unicyclic graphs.

Let $G$ be a graph and $S \subset V(G)$ such that $\deg_{G[S]}(u)=\deg_G(u)$ for all but exactly one vertex $x$ in $S$ where $\deg_{G[S]}(x)<\deg_G(x)$. We call $S$ a \textit{well-covered branch (induced by $x$)} if a well-covered tree can obtained by adding a leaf to $x$ in $G[S]$ and $S$ is the maximal set having these properties. For example, in each of the graphs shown in Figure \ref{fig:S345} the vertices of $T_1^{\ast}, \ldots, T_k^{\ast}$ together with $x$ form a well-covered branch induced by $x$. Note that in Figure \ref{fig:S345} (b), the leaf adjacent to $x$ is not included in the well-covered branch induced by $x$. Also note that $G[S - \{x\}] \cong F^{\ast}$ for some forest $F$. We say \emph{$x$ induces a well-covered branch isomorphic to $F^{\ast}$} if $G[S - \{x\}] \cong F^{\ast}$. Furthermore note that $x$ is not adjacent to any leaves of $F^{\ast}$. 

Let $\mathcal{S}_3$ be the class of all graphs that can be obtained from a $C_3$ by allowing exactly 1 or 2 vertices in the $3$-cycle to induce any well-covered branch. Let $\mathcal{S}_4$ be the class of all graphs of the form $(T\cup K_2)+au+bv$ where $T$ is a well-covered tree, $uv$ is a non-pendant edge of $T$ and the vertex set of the $K_2$ component is $\{a,b\}$. Let $\mathcal{S}_5$ be the class of all graphs that can be obtained from a $C_5$ by allowing exactly 1 or 2 nonadjacent vertices in the $5$-cycle to induce any well-covered branch. See Figure~\ref{fig:S345} for some general examples of these graphs. The graph classes $\mathcal{S}_3$, $\mathcal{S}_{4}$, and $\mathcal{S}_5$ were first defined in \cite{ToppVolkmann1990}, where the following characterization theorem was also given.

\setcounter{subfigure}{0}
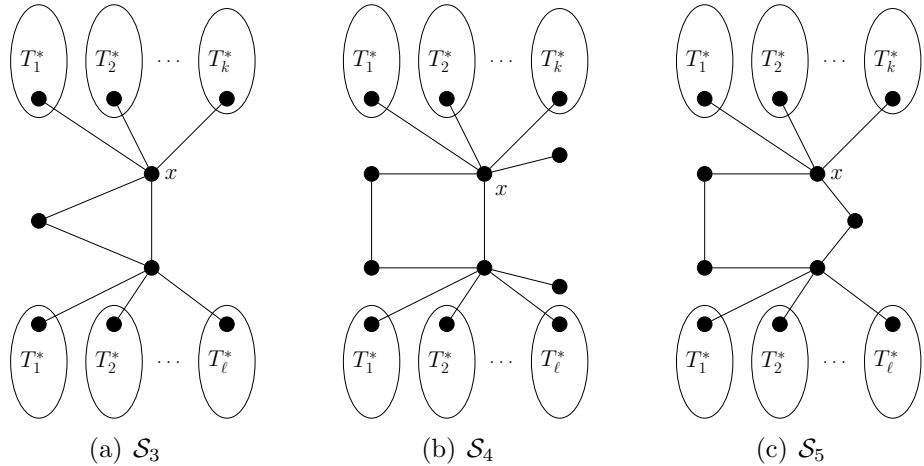
\begin{figure}[!h]
\def\c{0.5}
\def\r{1}
\centering
\subfigure[$\mathcal{S}_3$]{
\scalebox{\c}{
\begin{tikzpicture}
\draw (-3,4) ellipse (0.75 and 1.5);
\node[text width=1cm] at (-3,4) {\Large $T^{\ast}_{1}$};
\node[shape=circle,draw=black,fill=black] (u1) at (-3,3) {};

\draw (-1,4) ellipse (0.75 and 1.5);
\node[text width=1cm] at (-1,4) {\Large $T^{\ast}_{2}$};
\node[shape=circle,draw=black,fill=black] (u2) at (-1,3) {};

\draw (2,4) ellipse (0.75 and 1.4);
\node[text width=1cm] at (2,4) {\Large $T^{\ast}_{k}$};
\node[shape=circle,draw=black,fill=black] (uk) at (2,3) {};

\node[shape=circle,draw=black,fill=black] (r2) at (0,-1.5) {};
\node[shape=circle,draw=black,fill=black,label=right:\Large $x$] (r1) at (0,1) {};
\node[shape=circle,draw=black,fill=black] (r3) at (-3,-0.25) {};

\draw (-3,-4) ellipse (0.75 and 1.5);
\node[text width=1cm] at (-3,-4) {\Large $T^{\ast}_{1}$};
\node[shape=circle,draw=black,fill=black] (v1) at (-3,-3) {};

\draw (-1,-4) ellipse (0.75 and 1.5);
\node[text width=1cm] at (-1,-4) {\Large $T^{\ast}_{2}$};
\node[shape=circle,draw=black,fill=black] (v2) at (-1,-3) {};

\draw (2,-4) ellipse (0.75 and 1.5);
\node[text width=1cm] at (2,-4) {\Large $T^{\ast}_{\ell}$};
\node[shape=circle,draw=black,fill=black] (vk) at (2,-3) {};

\begin{scope}
    \path [-] (r1) edge node {} (r2);
    \path [-] (r3) edge node {} (r2);
    \path [-] (r1) edge node {} (r3); 
    \path [-] (r2) edge node {} (v1);
    \path [-] (r2) edge node {} (v2);
    \path [-] (r2) edge node {} (vk); 
    \path [-] (r1) edge node {} (u1);
    \path [-] (r1) edge node {} (u2);
    \path [-] (r1) edge node {} (uk); 
\end{scope}

\begin{scope}
    \path (0,-4) -- node[auto=false]{\Large \ldots} (1,-4);
    \path (0,4) -- node[auto=false]{\Large \ldots} (1,4);
\end{scope}
\end{tikzpicture}}}
\qquad
\subfigure[$\mathcal{S}_4$]{
\scalebox{\c}{
\begin{tikzpicture}
\draw (-3,4) ellipse (0.75 and 1.5);
\node[text width=1cm] at (-3,4) {\Large $T^{\ast}_{1}$};
\node[shape=circle,draw=black,fill=black] (u1) at (-3,3) {};

\draw (-1,4) ellipse (0.75 and 1.5);
\node[text width=1cm] at (-1,4) {\Large $T^{\ast}_{2}$};
\node[shape=circle,draw=black,fill=black] (u2) at (-1,3) {};

\draw (2,4) ellipse (0.75 and 1.4);
\node[text width=1cm] at (2,4) {\Large $T^{\ast}_{k}$};
\node[shape=circle,draw=black,fill=black] (uk) at (2,3) {};

\node[shape=circle,draw=black,fill=black] (r2) at (0,-1.5) {};
\node[shape=circle,draw=black,fill=black,label=below right:\Large $x$] (r1) at (0,1) {};
\node[shape=circle,draw=black,fill=black] (r3) at (-3,-1.5) {};
\node[shape=circle,draw=black,fill=black] (r4) at (-3,1) {};

\node[shape=circle,draw=black,fill=black] (l2) at (2,-2) {};
\node[shape=circle,draw=black,fill=black] (l1) at (2,1.5) {};

\draw (-3,-4) ellipse (0.75 and 1.5);
\node[text width=1cm] at (-3,-4) {\Large $T^{\ast}_{1}$};
\node[shape=circle,draw=black,fill=black] (v1) at (-3,-3) {};

\draw (-1,-4) ellipse (0.75 and 1.5);
\node[text width=1cm] at (-1,-4) {\Large $T^{\ast}_{2}$};
\node[shape=circle,draw=black,fill=black] (v2) at (-1,-3) {};

\draw (2,-4) ellipse (0.75 and 1.5);
\node[text width=1cm] at (2,-4) {\Large $T^{\ast}_{\ell}$};
\node[shape=circle,draw=black,fill=black] (vk) at (2,-3) {};

\begin{scope}
    \path [-] (r1) edge node {} (r2);
    \path [-] (r3) edge node {} (r2);
    \path [-] (r1) edge node {} (r4); 
    \path [-] (r3) edge node {} (r4); 
    \path [-] (r2) edge node {} (v1);
    \path [-] (r2) edge node {} (v2);
    \path [-] (r2) edge node {} (vk); 
    \path [-] (r1) edge node {} (u1);
    \path [-] (r1) edge node {} (u2);
    \path [-] (r1) edge node {} (uk); 
    
    \path [-] (r1) edge node {} (l1);
    \path [-] (r2) edge node {} (l2);
\end{scope}

\begin{scope}
    \path (0,-4) -- node[auto=false]{\Large \ldots} (1,-4);
    \path (0,4) -- node[auto=false]{\Large \ldots} (1,4);
\end{scope}
\end{tikzpicture}}}
\qquad
\subfigure[$\mathcal{S}_5$]{
\scalebox{\c}{
\begin{tikzpicture}
\draw (-3,4) ellipse (0.75 and 1.5);
\node[text width=1cm] at (-3,4) {\Large $T^{\ast}_{1}$};
\node[shape=circle,draw=black,fill=black] (u1) at (-3,3) {};

\draw (-1,4) ellipse (0.75 and 1.5);
\node[text width=1cm] at (-1,4) {\Large $T^{\ast}_{2}$};
\node[shape=circle,draw=black,fill=black] (u2) at (-1,3) {};

\draw (2,4) ellipse (0.75 and 1.4);
\node[text width=1cm] at (2,4) {\Large $T^{\ast}_{k}$};
\node[shape=circle,draw=black,fill=black] (uk) at (2,3) {};

\node[shape=circle,draw=black,fill=black] (r2) at (0,-1.5) {};
\node[shape=circle,draw=black,fill=black,label=right:\Large $x$] (r1) at (0,1) {};
\node[shape=circle,draw=black,fill=black] (r3) at (-3,-1.5) {};
\node[shape=circle,draw=black,fill=black] (r4) at (-3,1) {};
\node[shape=circle,draw=black,fill=black] (r5) at (1,-0.25) {};

\draw (-3,-4) ellipse (0.75 and 1.5);
\node[text width=1cm] at (-3,-4) {\Large $T^{\ast}_{1}$};
\node[shape=circle,draw=black,fill=black] (v1) at (-3,-3) {};

\draw (-1,-4) ellipse (0.75 and 1.5);
\node[text width=1cm] at (-1,-4) {\Large $T^{\ast}_{2}$};
\node[shape=circle,draw=black,fill=black] (v2) at (-1,-3) {};

\draw (2,-4) ellipse (0.75 and 1.5);
\node[text width=1cm] at (2,-4) {\Large $T^{\ast}_{\ell}$};
\node[shape=circle,draw=black,fill=black] (vk) at (2,-3) {};

\begin{scope}
    \path [-] (r1) edge node {} (r5);
    \path [-] (r5) edge node {} (r2);
    \path [-] (r3) edge node {} (r2);
    \path [-] (r1) edge node {} (r4); 
    \path [-] (r3) edge node {} (r4); 
    \path [-] (r2) edge node {} (v1);
    \path [-] (r2) edge node {} (v2);
    \path [-] (r2) edge node {} (vk); 
    \path [-] (r1) edge node {} (u1);
    \path [-] (r1) edge node {} (u2);
    \path [-] (r1) edge node {} (uk); 
\end{scope}

\begin{scope}
    \path (0,-4) -- node[auto=false]{\Large \ldots} (1,-4);
    \path (0,4) -- node[auto=false]{\Large \ldots} (1,4);
\end{scope}
\end{tikzpicture}}}
\caption{The general form of graphs in $\mathcal{S}_3$, $\mathcal{S}_4$, and $\mathcal{S}_5$ where $T_i$ and $T_i'$ are trees for all $i$.}
\label{fig:S345}
\end{figure}

\begin{theorem}[\cite{ToppVolkmann1990}]\label{thm:wcunicyclicgraphs}
A graph $G$ is a connected well-covered unicyclic graph if and only if $$G\in \{C_3,C_4,C_5,C_7\}\cup \mathcal{S}_3\cup \mathcal{S}_4\cup \mathcal{S}_5\cup\mathcal{KU}.$$
\end{theorem}

Note that all graphs in $\mathcal{S}_3$ and $\mathcal{S}_5$ have odd order while all graphs in $\mathcal{S}_4$ have even order by definition.

\subsubsection{Even order}

For our results in this section we will require some technical lemmas. The first is a result similar to Theorem~\ref{thm:boldstaroperation} that allows us to make small local changes to a graph that results in a graph which is smaller with respect to $\preceq$. Unlike Theorem~\ref{thm:boldstaroperation} though, this new operation will allow for the preservation of well-coveredness. 

Let $G$ be a graph with a well-covered branch $S$ induced by $x$ such that there are at least two vertices $a,b\in S-\{x\}$ such that $\deg_{G}(a)=\deg_{G}(b)=2$. Note that if there is exactly one vertex whose degree is $2$, then $G[S - \{x\}]\cong P_{|S|/2}^{\ast}$ and the unique neighbour of $x$ in $S-\{x\}$ has degree $3$ in $G$. Let $u,v\in S$ such that $\deg_G(u)=2$ and $v$ is a vertex with $\deg_{G[S]}(v) \geq 4$ at shortest distance to $u$ among all vertices with degree at least $4$ in $G[S]$. If no such $v$ exists or $v$ is further from $u$ than $x$, then assign $v=x$. Note if $v=x$ then $\deg_{G[S]}(x) \geq 2$ and therefore $\deg_{G}(x) \geq 3$. Let $w\in S$ be a non-leaf neighbour of $v$ not on the path from $u$ to $x$. We define the \textit{dagger operation} on the graph $G$, denoted $G^{\dagger}_{u,v,w}$, to be the the graph $G-vw+uw$. Note that $S$ is still a well-covered branch in $G^{\dagger}_{u,v,w}$. See Figure \ref{fig:Gdagger} for an illustration of this operation. 

\setcounter{subfigure}{0}
\begin{figure}[!h]
\def\c{0.5}
\def\r{1}
\centering
\subfigure[$S$ in $G$]{
\scalebox{\c}{
\begin{tikzpicture}
\draw (0.5,-0.25) ellipse (1.25 and 2.25);
\node[text width=1cm] at (0.5,1.25) {\Large $F^{\ast}$};

\node[shape=circle,draw=black,fill=black, label=left:\Large $u$] (1) at (-6,0) {};
\node[shape=circle,draw=black,fill=black] (2) at (-4,0) {};
\node[shape=circle,draw=black,fill=black, label=above right:\Large $v$] (3) at (-2,0) {};

\node[shape=circle,draw=black,fill=black] (1l) at (-6,-2) {};
\node[shape=circle,draw=black,fill=black] (2l) at (-4,-2) {};
\node[shape=circle,draw=black,fill=black] (3l) at (-2,-2) {};

\draw (-3.5,1.5) ellipse (1.5 and 0.75);
\node[text width=1cm] at (-4,1.5) {\Large $T^{\ast}$};
\node[shape=circle,draw=black,fill=black, label=above:\Large $w$] (w) at (-3,1.25) {};

\node[shape=circle,draw=black,fill=black] (4) at (0,0.25) {};
\node[shape=circle,draw=black,fill=black] (5) at (0,-0.5) {};
\node[shape=circle,draw=black,fill=black] (6) at (0,-1.5) {};

\node[shape=circle,draw=black,fill=black] (7) at (1.25,0.75) {};
\node[shape=circle,draw=black,fill=black] (8) at (1.25,0) {};
\node[shape=circle,draw=black,fill=black] (9) at (1.25,-1) {};
\node[shape=circle,draw=black,fill=black, label=above:\Large $x$] (x) at (3.25,-0.25) {};

\begin{scope}
    \path [-] (1) edge node {} (2);
    \path [-] (2) edge node {} (-3.5,0);
    \path [-] (-2.5,0) edge node {} (3);
    \path [-] (3) edge node {} (4);
    \path [-] (3) edge node {} (5);
    \path [-] (3) edge node {} (6);
    
    \path [-] (1) edge node {} (1l);
    \path [-] (2) edge node {} (2l);
    \path [-] (3) edge node {} (3l);
    
    \path [-] (3) edge node {} (w);
    
    \path [-] (x) edge node {} (7);
    \path [-] (x) edge node {} (8);
    \path [-] (x) edge node {} (9);
       
\end{scope}

\begin{scope}
    \path (-3.5,0) -- node[auto=false]{\Large \ldots} (-2.5,0);
    \path (5) -- node[auto=false]{\Large \vdots} (6);
    \path (8) -- node[auto=false]{\Large \vdots} (9);
\end{scope}
\end{tikzpicture}}}
\qquad
\subfigure[$S$ in $G^{\dagger}_{u,v,w}$]{
\scalebox{\c}{
\begin{tikzpicture}
\draw (0.5,-0.25) ellipse (1.25 and 2.25);
\node[text width=1cm] at (0.5,1.25) {\Large $F^{\ast}$};

\node[shape=circle,draw=black,fill=black, label=left:\Large $u$] (1) at (-6,0) {};
\node[shape=circle,draw=black,fill=black] (2) at (-4,0) {};
\node[shape=circle,draw=black,fill=black, label=above right:\Large $v$] (3) at (-2,0) {};

\node[shape=circle,draw=black,fill=black] (1l) at (-6,-2) {};
\node[shape=circle,draw=black,fill=black] (2l) at (-4,-2) {};
\node[shape=circle,draw=black,fill=black] (3l) at (-2,-2) {};

\draw (-3.5,1.5) ellipse (1.5 and 0.75);
\node[text width=1cm] at (-4,1.5) {\Large $T^{\ast}$};
\node[shape=circle,draw=black,fill=black, label=above:\Large $w$] (w) at (-3,1.25) {};

\node[shape=circle,draw=black,fill=black] (4) at (0,0.25) {};
\node[shape=circle,draw=black,fill=black] (5) at (0,-0.5) {};
\node[shape=circle,draw=black,fill=black] (6) at (0,-1.5) {};

\node[shape=circle,draw=black,fill=black] (7) at (1.25,0.75) {};
\node[shape=circle,draw=black,fill=black] (8) at (1.25,0) {};
\node[shape=circle,draw=black,fill=black] (9) at (1.25,-1) {};
\node[shape=circle,draw=black,fill=black, label=above:\Large $x$] (x) at (3.25,-0.25) {};

\begin{scope}
    \path [-] (1) edge node {} (2);
    \path [-] (2) edge node {} (-3.5,0);
    \path [-] (-2.5,0) edge node {} (3);
    \path [-] (3) edge node {} (4);
    \path [-] (3) edge node {} (5);
    \path [-] (3) edge node {} (6);
    
    \path [-] (1) edge node {} (1l);
    \path [-] (2) edge node {} (2l);
    \path [-] (3) edge node {} (3l);
    
    \path [-] (1) edge node {} (w);
    
    \path [-] (x) edge node {} (7);
    \path [-] (x) edge node {} (8);
    \path [-] (x) edge node {} (9);
       
\end{scope}

\begin{scope}
    \path (-3.5,0) -- node[auto=false]{\Large \ldots} (-2.5,0);
    \path (5) -- node[auto=false]{\Large \vdots} (6);
    \path (8) -- node[auto=false]{\Large \vdots} (9);
\end{scope}
\end{tikzpicture}}}
\caption{Illustration of the dagger operation.}%
\label{fig:Gdagger}%
\end{figure}
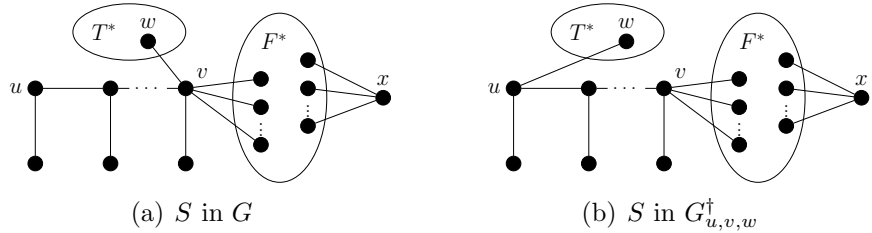

\begin{lemma}\label{lem:dagger}
If $G$ is a graph as above, then $G^{\dagger}_{u,v,w}\preceq G$.
\end{lemma}
\begin{proof}
Let $G$ be a graph with $S$ a well-covered branch induced by $x$ such that there are at least two vertices $a,b\in S-\{x\}$ such that $\deg_{G}(a)=\deg_{G}(b)=2$. Let and $u,v,w\in V(G)$ as required in the definition of the dagger operation. Let $T$ be the collection of vertices in $S$ whose shortest path to $u$ does not contain $v$. Note $G[T] \cong P_{k}^{\ast}$ for some $k>0$. Also let $H=G-T$. It is clear that $G-vw=G^{\dagger}_{u,v,w}-uw$, so $G-vw\succeq G^{\dagger}_{u,v,w}-uw$. Now by Theorem~\ref{thm:recursivepartialorder}, it suffices to show $G-N[vw]\preceq G^{\dagger}_{u,v,w}-N[uw]$. Furthermore by Theorem \ref{thm:subgraphlessthangraph} it suffices to show $G^{\dagger}_{u,v,w}-N[uw]$ has a subgraph isomorphic to  $G-N[vw]\cong P_{k-1}^{\ast}\cup K_1\cup (H-N[v]\cup N[w])$. Clearly $H-N[v]\cup N[w]$ is a subgraph of $G^{\dagger}_{u,v,w}-N[uw]$. Now with the remaining vertices we can obtain $P_{k-1}^{\ast} \cup K_1$ from the $P_{k-2}^{\ast} \cup K_1$ left from $T$ after deleting $N[u]$ together with $v$ and one of its neighbours not in $T$. Note if $v \neq x$, $v$ necessarily has a leaf neighbour not in $T$, and if $v = x$ then $x$ has a neighbour not in $S$ which was otherwise deleted in $G-N[vw]$. Therefore $G-N[vw]$ is a subgraph of $G^{\dagger}_{u,v,w}-N[uw]$.
\end{proof}

Let $G$ be a graph with well-covered branch $S$ induced by $x$. Note $G^{\dagger}_{u,v,w}$ reduces the number of degree $2$ vertices in $S$ while keeping it a well-covered branch. It follows that repeated applications of the previous lemma allows the $S-x$ to be replaced by $P_{n}^{\ast}$ for the appropriate value of $n$. This gives the next technical lemma.

\begin{lemma}\label{lem:replacebranchwithpath}
Let $G$ be a connected well-covered unicyclic graph with a well-covered branch $S$ induced by $x$. Let the graph $H$ be obtained from $G$ by rearranging the edges induced by $S-\{x\}$ such that $x$ now induces the well-covered branch $S$ such that $\deg_{H}(u)=2$ for exactly one vertex $v\in S-\{x\}$. Then $H\preceq G$.
\end{lemma}

\begin{proof}
We shall induct on the number of degree 2 vertices in $S-\{x\}$. Note that by definition $\deg_{G[S]}(v)=\deg_{G}(v)$ for all $v \in S-\{x\}$. If $S-\{x\}$ has one vertex of degree two, then $H \cong G$ so $H\preceq G$, trivially. Now suppose $S-\{x\}$ has $k > 1$ vertices of degree two. Therefore there exists $u,v,w \in S$ as in the definition of the dagger operation to obtain $G^{\dagger}_{u,v,w}$. By Lemma  \ref{lem:dagger}, $G^{\dagger}_{u,v,w}\preceq G$. The only vertices in $S-\{x\}$ with different degrees in $G$ and $G^{\dagger}_{u,v,w}$ are $u$ and, if $v\neq x$, $v$. Furthermore $u$ and $v$ both have degree at least three in $G^{\dagger}_{u,v,w}$, so there are fewer vertices of degree $2$ in $S-\{x\}$ in $G^{\dagger}_{u,v,w}$ than there are in $S-\{x\}$ in $G$. Therefore our inductive hypothesis applies to give $H \preceq  G^{\dagger}_{u,v,w}\preceq G$. 
\end{proof}

\begin{theorem}
If $G$ is a connected well-covered unicyclic graph of order $2n$, $n\ge 3$, then 
$$C_n^{\ast}\preceq G\preceq U_n^{\ast}.$$
\end{theorem}
\begin{proof}
Suppose $G\in \mathcal{KU}$, i.e. $G=H^\ast$ for some connected unicyclic graph $H$, then the result follows by Lemma~\ref{lem:relationresptostar} and Theorem~\ref{thm:maxminunicyclic}. Now suppose $G\notin \mathcal{KU}$, then by Theorem~\ref{thm:wcunicyclicgraphs} $G\in \mathcal{S}_4$.

We first show that $G\preceq U_n^{\ast}$. Let $a,b,u,w$ be the vertices of the cycle in $G$ such that $N_G(a)=\{u,b\}$ and $N_G(b)=\{a,w\}$. Let $F$ be the graph $G-bw+aw$. Therefore, $F^{\star}_{b,a,w}=G$ so by Theorem~\ref{thm:boldstaroperation}, 
\begin{align}
G\preceq F.\label{eq:wcuni1}
\end{align}
 Also note that $F=H^\ast$ for some connected unicyclic graph of order $n$ and girth $3$, so by Lemma~\ref{lem:relationresptostar} and Theorem~\ref{thm:maxminunicyclic}, 
\begin{align}
F\preceq U_n^{\ast}.\label{eq:wcuni2}
\end{align}

\noindent By Theorem~\ref{thm:partialorder} and (\ref{eq:wcuni1}) and (\ref{eq:wcuni2}), it follows that $G\preceq U_n^{\ast}$.

We now show that $C_n^{\ast}\preceq G$. Recall $G\in \mathcal{S}_4$. By Lemma~\ref{lem:replacebranchwithpath} it suffices to assume $G$ is isomorphic to $(K_2\cup P_{k}^{\ast}\cup P_{\ell}^{\ast})+ua+vb+uv$ where $v$ is a degree $2$ vertex in $P_{\ell}^{\ast}$, $u$ is a degree $2$ vertex in $P_{k}^{\ast}$, and the vertex set of $K_2$ is equal to $\{a,b\}$. Note that $k+\ell=n-1$. Let $f\in E(C_n^{\ast})$ be an edge on the cycle. Now $G-au$ is a well-covered tree of order $2n$ and $C_n^{\ast}-f=P_n^{\ast}$, so by Corollary~\ref{cor:wctreebounds}, $C_n^{\ast}-f\preceq G-au$. On the other hand, $G-N[au]\cong P_{\ell-1}^{\ast}\cup P_{k-2}^{\ast}\cup \overline{K_2}$, which is clearly a subgraph of $P_{n-4}^{\ast}\cup \overline{K_2}\cong C_n^{\ast}-N[f]$. So by Theorem~\ref{thm:subgraphlessthangraph} and Theorem~\ref{thm:recursivepartialorder}, $C_n^{\ast}\preceq G$.

\end{proof}

\begin{corollary}
If $G$ is a connected well-covered unicyclic graph of order $2n$, $n\ge 3$, then 
$$\xi(C_n^{\ast})\le \xi(G)\le \xi(U_n^{\ast}).$$
\end{corollary}

\subsubsection{Odd order}
For odd $n$, let $M_n$ be the graph obtained from $U_{\tfrac{n+1}{2}}$ by subdividing each of its pendant edges, see Figure~\ref{fig:Mn}.

 \begin{figure}[htb]
\def\c{0.5}
\def\r{2}
\centering
\scalebox{\c}{
\begin{tikzpicture}
\begin{scope}[every node/.style={circle,thick,draw,fill}]

     \node (13) at (-1*\r,-1*\r) {}; 
     \node (11) at (-1*\r,0*\r) {};   
    \node (6) at (-1*\r,0.75*\r) {};
    
    \node[label=left:\large $v_{\tfrac{n-1}{2}}$] (L3) at (-2*\r,-1*\r) {}; 
     \node[label=left:\large $v_2$](L2) at (-2*\r,0*\r) {};   
    \node[label=left:\large $v_1$](L1) at (-2*\r,0.75*\r) {};
    
    \node(5) at (0*\r,0*\r) {};
\node(4) at (1*\r,sqrt{2/3}*\r) {};
    \node (3) at (1*\r,-sqrt{2/3}*\r) {}; 
      
\end{scope}

\begin{scope}
   
    \path [-] (5) edge node {} (4);
    
    \path [-] (4) edge node {} (3);
    \path [-] (5) edge node {} (3);
    
    \path [-] (6) edge node {} (5);
    \path [-] (5) edge node {} (11);
    \path [-] (5) edge node {} (13);
    
      \path [-] (6) edge node {} (L1);
    \path [-] (L2) edge node {} (11);
    \path [-] (L3) edge node {} (13);
\end{scope}

\path (11) -- node[auto=false]{\vdots} (13);

\end{tikzpicture}}
\caption{The graph $M_{n}$.}%
\label{fig:Mn}%
\end{figure}
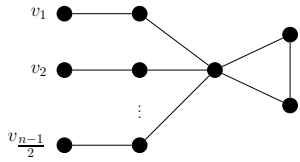

We require one more result on computing the independence polynomial before we can prove that $M_n$ is the maximum graph.

\begin{theorem}[\cite{Hoede1994}]\label{thm:cliquedeletionformula}
If $G$ is a graph with $C$ a complete subgraph, then 
$$I(G,x)=I(G-C,x)+x\sum_{v\in V(C)}I(G-N[v],x).$$
\end{theorem}

\begin{theorem}\label{thm:wcunicyclicoddupperbound}
If $n$ is odd and $G$ is a connected well-covered unicyclic graph of order $n$, then $G\preceq M_n$.
\end{theorem}
\begin{proof}
Let $n\ge 3$ be odd and $G$ be a connected well-covered unicyclic graph of order $n$. The proof is in cases depending on the number of vertices of degree greater than $2$ in the cycle of $G$.

\vspace{2mm}
\noindent\textbf{Case 1:} The cycle of $G$ has exactly $1$ vertex of degree greater than $2$. 
\vspace{2mm}

Suppose that $\operatorname{girth}(G)=3$ and let $v$ be a vertex of degree $2$ on the cycle. Let $u$ be a vertex of degree $2$ on the cycle in $M_n$. By definition of graphs in $\mathcal{S}_3$, $G-v$ is a well-covered tree of order $n-1$ and $M_n-u=S_{\frac{n-1}{2}}^{\ast}$. Hence, by Corollary~\ref{cor:wctreebounds}, $G-v\preceq M_n-u$. On the other hand, $G-N[v]$ is a well-covered forest of order $n-3$ and $M_n-N[u]=\frac{n-3}{2}K_2$. Thus, by Corollary~\ref{cor:wctreesaregraphstars}, $M_n-N[u]$ is a subgraph of $G-N[v]$. So by Theorem~\ref{thm:subgraphlessthangraph}, $M_n-N[u]\preceq G-N[v]$. Finally, by Theorem~\ref{thm:recursivepartialorder}, $G\preceq M_n$. We note that if $\operatorname{girth}(G)=5$, then a similar argument shows that $G\preceq M_n$ by choosing $u$ to be a vertex of degree $2$ with all neighbours of degree $2$ on the cycle of $G$.

\vspace{2mm}
\noindent\textbf{Case 2:} The cycle of $G$ has exactly $2$ vertices of degree greater than $2$. 
\vspace{2mm}

We first suppose that $\operatorname{girth}(G)=3$. Let $\{u,v,w\}$ be the vertices of the $3$-cycle in $G$ such that $\deg(w)=2$. From Theorem~\ref{thm:wcunicyclicgraphs}, $G$ has two well-covered branches $S_1$ and $S_2$ induced by $v$ and $u$, respectively. Let $G_1=G[S_1-\{v\}]$ and $G_2=G[S_2-\{u\}]$. Note $G_1=F_1^{\ast}$ and $G_2=F_2^{\ast}$ where $F_1$ and $F_2$ are forests. Suppose $|V(G_1)|=2k$ and $|V(G_2)|=2\ell$ for some positive integers $k$ and $\ell$. Now let $G'$ be the graph obtained from $G$ by deleting all non-pendant edges from $G_2$ and then joining a single vertex in every resulting $K_2$ component to $u$ (i.e. replacing $G_2$ with $kK_2$ and joining a vertex from each $K_2$ to $u$, see Figure~\ref{fig:GG'G''}). We now have, 
\begin{align}
I(G-v,x)-I(G'-v,x)&=I(G_1,x)\left(I(G[V(G_2)\cup\{w\}],x)-I(S_{\ell+1}^{\ast},x)\right). \label{eq:G-vprime}
\end{align}

\setcounter{subfigure}{0}
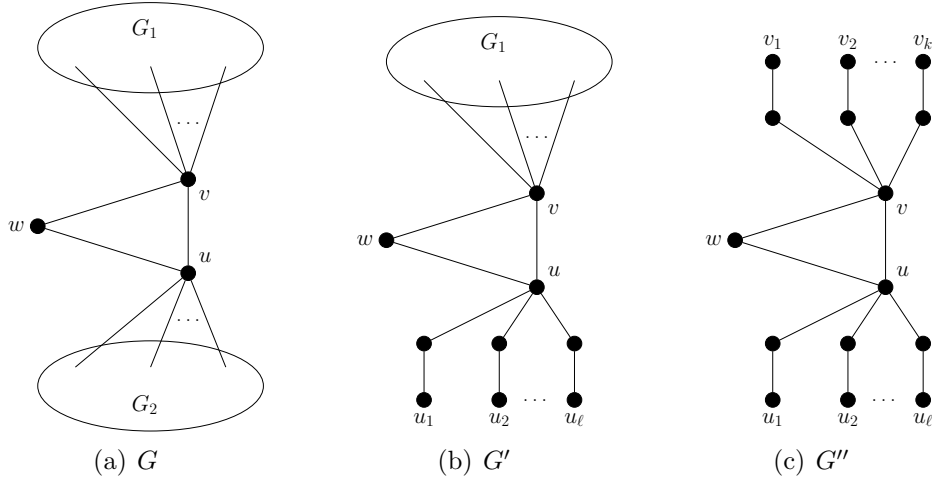
\begin{figure}[!h]
\def\c{0.5}
\def\r{1}
\centering
\subfigure[$G$]{
\scalebox{\c}{
\begin{tikzpicture}
\draw (-1,4.5) ellipse (3 and 1.2);
\node[text width=1cm] at (-1,5) {\Large $G_1$};

\node[shape=circle,draw=black,fill=black, label=above right:\Large $u$] (r2) at (0,-1.5) {};
\node[shape=circle,draw=black,fill=black,label=below right:\Large $v$] (r1) at (0,1) {};
\node[shape=circle,draw=black,fill=black,label=left:\Large $w$] (r3) at (-4,-0.25) {};

\draw (-1,-4.5) ellipse (3 and 1.2);
\node[text width=1cm] at (-1,-5) {\Large $G_2$};

\begin{scope}
    \path [-] (r1) edge node {} (r2);
    \path [-] (r3) edge node {} (r2);
    \path [-] (r1) edge node {} (r3); 
    \path [-] (r2) edge node {} (-3,-4);
    \path [-] (r2) edge node {} (-1,-4);
    \path [-] (r2) edge node {} (1,-4); 
    \path [-] (r1) edge node {} (-3,4);
    \path [-] (r1) edge node {} (-1,4);
    \path [-] (r1) edge node {} (1,4); 
\end{scope}

\begin{scope}
    \path (-1,-2.75) -- node[auto=false]{\Large \ldots} (1.1,-2.75);
    \path (-1,2.5) -- node[auto=false]{\Large \ldots} (1.1,2.5);
\end{scope}
\end{tikzpicture}}}
\qquad
\subfigure[$G'$]{
\scalebox{\c}{
\begin{tikzpicture}

\draw (-1,4.5) ellipse (3 and 1.2);
\node[text width=1cm] at (-1,5) {\Large $G_1$};

\node[shape=circle,draw=black,fill=black, label=above right:\Large $u$] (r2) at (0,-1.5) {};
\node[shape=circle,draw=black,fill=black,label=below right:\Large $v$] (r1) at (0,1) {};
\node[shape=circle,draw=black,fill=black,label=left:\Large $w$] (r3) at (-4,-0.25) {};

\node[shape=circle,draw=black,fill=black] (v1) at (-3,-3) {};
\node[shape=circle,draw=black,fill=black,label=below:\Large $u_{1}$] (u1) at (-3,-4.5) {};

\node[shape=circle,draw=black,fill=black] (v2) at (-1,-3) {};
\node[shape=circle,draw=black,fill=black,label=below:\Large $u_{2}$] (u2) at (-1,-4.5) {};

\node[shape=circle,draw=black,fill=black] (vk) at (1,-3) {};
\node[shape=circle,draw=black,fill=black,label=below:\Large $u_{\ell}$] (uL) at (1,-4.5) {};

\begin{scope}
    \path [-] (r1) edge node {} (r2);
    \path [-] (r3) edge node {} (r2);
    \path [-] (r1) edge node {} (r3); 
    \path [-] (r2) edge node {} (v1);
    \path [-] (r2) edge node {} (v2);
    \path [-] (r2) edge node {} (vk); 
    \path [-] (r1) edge node {} (-3,4);
    \path [-] (r1) edge node {} (-1,4);
    \path [-] (r1) edge node {} (1,4); 
    
    \path [-] (u1) edge node {} (v1);
    \path [-] (u2) edge node {} (v2);
    \path [-] (uL) edge node {} (vk);
\end{scope}

\begin{scope}
    \path (-1,-4.5) -- node[auto=false]{\Large \ldots} (1,-4.5);
    \path (-1,2.5) -- node[auto=false]{\Large \ldots} (1.1,2.5);
\end{scope}
\end{tikzpicture}}}
\qquad
\subfigure[$G''$]{
\scalebox{\c}{
\begin{tikzpicture}

\node[shape=circle,draw=black,fill=black] (1v1) at (-3,3) {};
\node[shape=circle,draw=black,fill=black,label=above:\Large $v_{1}$] (v11) at (-3,4.5) {};

\node[shape=circle,draw=black,fill=black] (2v2) at (-1,3) {};
\node[shape=circle,draw=black,fill=black,label=above:\Large $v_{2}$] (v22) at (-1,4.5) {};

\node[shape=circle,draw=black,fill=black] (kvk) at (1,3) {};
\node[shape=circle,draw=black,fill=black,label=above:\Large $v_{k}$] (vkk) at (1,4.5) {};

\node[shape=circle,draw=black,fill=black, label=above right:\Large $u$] (r2) at (0,-1.5) {};
\node[shape=circle,draw=black,fill=black,label=below right:\Large $v$] (r1) at (0,1) {};
\node[shape=circle,draw=black,fill=black,label=left:\Large $w$] (r3) at (-4,-0.25) {};

\node[shape=circle,draw=black,fill=black] (v1) at (-3,-3) {};
\node[shape=circle,draw=black,fill=black,label=below:\Large $u_{1}$] (u1) at (-3,-4.5) {};

\node[shape=circle,draw=black,fill=black] (v2) at (-1,-3) {};
\node[shape=circle,draw=black,fill=black,label=below:\Large $u_{2}$] (u2) at (-1,-4.5) {};

\node[shape=circle,draw=black,fill=black] (vk) at (1,-3) {};
\node[shape=circle,draw=black,fill=black,label=below:\Large $u_{\ell}$] (uL) at (1,-4.5) {};

\begin{scope}
    \path [-] (r1) edge node {} (r2);
    \path [-] (r3) edge node {} (r2);
    \path [-] (r1) edge node {} (r3); 
    \path [-] (r2) edge node {} (v1);
    \path [-] (r2) edge node {} (v2);
    \path [-] (r2) edge node {} (vk); 
    \path [-] (r1) edge node {} (1v1);
    \path [-] (r1) edge node {} (2v2);
    \path [-] (r1) edge node {} (kvk); 
    
    \path [-] (u1) edge node {} (v1);
    \path [-] (u2) edge node {} (v2);
    \path [-] (uL) edge node {} (vk);
    
    \path [-] (v11) edge node {} (1v1);
    \path [-] (v22) edge node {} (2v2);
    \path [-] (vkk) edge node {} (kvk);
\end{scope}

\begin{scope}
    \path (-1,-4.5) -- node[auto=false]{\Large \ldots} (1,-4.5);
    \path (-1,4.5) -- node[auto=false]{\Large \ldots} (1.1,4.5);
\end{scope}
\end{tikzpicture}}}
\caption{Graphs used in the proof of Theorem~\ref{thm:wcunicyclicoddupperbound}.}%
\label{fig:GG'G''}%
\end{figure}

It is clear that $G[V(G_2)\cup\{w\}]$ is a well-covered tree of order $2\ell+2$, so by Corollary~\ref{cor:wctreebounds} it follows that $G[V(G_2)\cup\{w\}]\preceq S_{\ell+1}^{\ast}$. Since $G_1$ is a subgraph of $G'-v$, it follows that $\xi(G_1)\le \xi(G'-v)$, so $I(G_1,x)\ge 0$ for all $x\in [\xi(G'-v),0]$. Therefore, (\ref{eq:G-vprime}) is at least $0$ so $G-v\preceq G'-v$. Further, 
\begin{align}
I(G-N[v],x)-I(G'-N[v],x)&=I(G_1-N[v],x)\left(I(G_2,x)-I(\ell K_2,x)\right). \label{eq:G-[v]prime}
\end{align}

Since $G_1-N[v]$ is a subgraph of $G'-N[v]$, it follows that $\xi(G_1-N[v])\le \xi(G'-N[v])$, so $I(G_1-N[v],x)\ge 0$ for all $x\in [\xi(G'-N[v]),0]$. Also, $G_2$ has a perfect matching, so $\ell K_2$ is a subgraph of $G_2$. So by Theorem~\ref{thm:subgraphlessthangraph}, $I(G_2,x)-I(\ell K_2,x)\le 0$ for all $x\in [\xi(G'-N[v],0]$. Thus, from (\ref{eq:G-[v]prime}), $G'-N[v]\preceq G-N[v]$.  Finally, we can conclude that $G\preceq G'$ by Theorem~\ref{thm:recursivepartialorder}. 

Let $G''$ be the graph obtained from $G'$ by deleting all non-pendant edges from $G_1$ and then joining a single vertex in any resulting $K_2$ components to $v$, see Figure~\ref{fig:GG'G''}. From the work done in the last few paragraphs, it follows that $G'\preceq G''$. 

To complete the proof, all that remains is to show that $G''\preceq M_n$. We apply Theorem~\ref{thm:cliquedeletionformula} to $G''$ and $M_n$ where the clique chosen in each case is the unique triangle. With these reductions it follows that for all $x\in [\xi(M_n),0]$,
$$
I(G'',x)-I(M_n,x)=x\left(I(\ell K_2,x)-I(\ell K_1,x) \right)\left(I(k K_1,x)-I(k K_2,x) \right) \ge 0.
$$
Therefore, $G''\preceq M_n$.

Except for the application of Theorem~\ref{thm:cliquedeletionformula}, every step in Case 2 can be done similarly if $\operatorname{girth}(G)=5$. In other words, it suffices to show for every graph $F\in \mathcal{S}_5$ of the form depicted in Figure~\ref{fig:girth5final} that $F\preceq G''$ where $G''$ is the graph from earlier in the proof, see Figure~\ref{fig:GG'G''}. However, $G^{\prime\prime\star}_{v_k,v,y}=F$, so $F\preceq G''$ by Theorem~\ref{thm:boldstaroperation}.

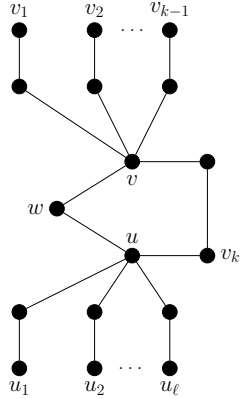
\begin{figure}[htb]
\def\c{0.5}
\centering
\scalebox{\c}{
\begin{tikzpicture}

\begin{scope}
	\node[shape=circle,draw=black,fill=black] (1v1) at (-3,3) {};
	\node[shape=circle,draw=black,fill=black,label=above:\Large $v_{1}$] (v11) at (-3,4.5) {};

	\node[shape=circle,draw=black,fill=black] (2v2) at (-1,3) {};
	\node[shape=circle,draw=black,fill=black,label=above:\Large $v_{2}$] (v22) at (-1,4.5) {};

	\node[shape=circle,draw=black,fill=black] (kvk) at (1,3) {};
	\node[shape=circle,draw=black,fill=black,label=above:\Large $v_{k-1}$] (vkk) at (1,4.5) {};

	\node[shape=circle,draw=black,fill=black,label=below:\Large $v$] (r1) at (0,1) {};
	\node[shape=circle,draw=black,fill=black,label=left:\Large $w$] (r3) at (-2,-0.25) {};
	\node[shape=circle,draw=black,fill=black, label=above:\Large $u$] (r2) at (0,-1.5) {};
	\node[shape=circle,draw=black,fill=black,label=right:\Large $v_k$] (r4) at (2,-1.5) {};
	\node[shape=circle,draw=black,fill=black] (r5) at (2,1) {};

	\node[shape=circle,draw=black,fill=black] (v1) at (-3,-3) {};
	\node[shape=circle,draw=black,fill=black,label=below:\Large $u_{1}$] (u1) at (-3,-4.5) {};

	\node[shape=circle,draw=black,fill=black] (v2) at (-1,-3) {};
	\node[shape=circle,draw=black,fill=black,label=below:\Large $u_{2}$] (u2) at (-1,-4.5) {};

	\node[shape=circle,draw=black,fill=black] (vk) at (1,-3) {};
	\node[shape=circle,draw=black,fill=black,label=below:\Large $u_{\ell}$] (uL) at (1,-4.5) {};
\end{scope}

\begin{scope}
    \path [-] (r1) edge node {} (r5);
    \path [-] (r4) edge node {} (r5);
    \path [-] (r4) edge node {} (r2);
    \path [-] (r3) edge node {} (r2);
    \path [-] (r1) edge node {} (r3); 
    \path [-] (r2) edge node {} (v1);
    \path [-] (r2) edge node {} (v2);
    \path [-] (r2) edge node {} (vk); 
    \path [-] (r1) edge node {} (1v1);
    \path [-] (r1) edge node {} (2v2);
    \path [-] (r1) edge node {} (kvk); 
    
    \path [-] (u1) edge node {} (v1);
    \path [-] (u2) edge node {} (v2);
    \path [-] (uL) edge node {} (vk);
    
    \path [-] (v11) edge node {} (1v1);
    \path [-] (v22) edge node {} (2v2);
    \path [-] (vkk) edge node {} (kvk);
\end{scope}

\begin{scope}
    \path (-1,-4.5) -- node[auto=false]{\Large \ldots} (1,-4.5);
    \path (-1,4.5) -- node[auto=false]{\Large \ldots} (1.1,4.5);
\end{scope}
\end{tikzpicture}}
\caption{$F\in\mathcal{S}_5$.}\label{fig:girth5final}
\end{figure}

\end{proof}

Let $G(g,k,\ell)$ be a graph obtained from a cycle $C=\{c_1,c_2,\ldots c_{g}\}$, $g=3$ or $g=5$, where $c_1$ and $c_3$ each induce well-covered branches isomorphic to $P_{k}^{\ast}$ and $P_{\ell}^{\ast}$ respectively and the unique neighbour of $c_1$ in the $P_{k}^{\ast}$ has degree $3$ in $G(g,k,\ell)$, respectively for $c_3$ and its unique neighbour in the $P_{\ell}^{\ast}$. We require that $k>0$ or $\ell>0$. Let $\mathcal{S}_P$ be the set of all $G(g,k,\ell)$ for every combination of $g,k,\ell$.  Note that $\mathcal{S}_P\subseteq \mathcal{S}_3\cup \mathcal{S}_5$. Examples of graphs in $\mathcal{S}_P$ can be found throughout the proof of Lemma \ref{lem:equivofoddwcuniminimal} in Figures~\ref{fig:BaseCase1} to \ref{fig:5Casek}.

We will now show that every graph of equal order in $\mathcal{S}_P$ has the same independence polynomial. First we will require the following result.

\begin{prop}[\cite{INDFIRST}]\label{prop:deletion}
Let $G$ and $H$ be graphs and $v\in V(G)$. Then
\begin{itemize}
\item[(i)] $I(G,x)=I(G-v,x)+xI(G-N[v],x)$,
\item[(ii)] $I(G\cup H,x)=I(G,x)\cdot I(H,x)$.
\end{itemize}
\end{prop}

\begin{lemma}\label{lem:equivofoddwcuniminimal}
Let $G ,H \in \mathcal{S}_P$. If $G$ and $H$ have the number of vertices then $I(G,x)=I(H,x).$
\end{lemma}

\begin{proof}
Recall $G\sim H$ denotes that $G$ and $H$ are independence equivalent. Note that for vertices $u \in V(G)$ and $u \in V(H)$, if $G-u \sim H-v$ and $G-N[u] \sim H-N[v]$ then $G \sim H$ by Proposition \ref{prop:deletion}. We will begin by showing $G(3,k+\ell,0) \sim G(3,k,\ell)$ for all $k,\ell>0$. It is sufficient to show for every fixed $k>0$, $G(3,k+\ell,0) \sim G(3,k,\ell)$ for all $\ell > 0$. We will show this through induction on $\ell$. 

Fix a $k>0$ and let $\ell = 1$. Let $v_1$ and $u_{k+1}$ be the vertices in $G(3,k,1)$ and $G(3,k+1,0)$, respectively illustrated in Figure~\ref{fig:BaseCase1}.

\setcounter{subfigure}{0}
\begin{figure}[!h]
\def\c{0.45}
\def\r{1}
\centering
\subfigure[$G(3,k,1)$]{
\scalebox{\c}{
\begin{tikzpicture}
\node[shape=circle,draw=black,fill=black,label=below:\Large $c_{1}$] (1) at (0,0) {};
\node[shape=circle,draw=black,fill=black,label=left:\Large $c_{2}$] (2) at (1,2) {};
\node[shape=circle,draw=black,fill=black,label=below:\Large $c_{3}$] (3) at (2,0) {};

\node[shape=circle,draw=black,fill=black,label=below:\Large $u_{1}$] (4) at (4,0) {};
\node[shape=circle,draw=black,fill=black] (4l) at (4,2) {};
\node[shape=circle,draw=black,fill=black,label=below:\Large $u_{2}$] (5) at (6,0) {};
\node[shape=circle,draw=black,fill=black] (5l) at (6,2) {};
\node[shape=circle,draw=black,fill=black,label=below:\Large $u_{k}$] (6) at (10,0) {};
\node[shape=circle,draw=black,fill=black] (6l) at (10,2) {};
\node[shape=circle,draw=black,fill=black,label=below:\Large $v_{1}$] (7) at (-2,0) {};
\node[shape=circle,draw=black,fill=black] (7l) at (-2,2) {};

\begin{scope}
    \path [-] (1) edge node {} (2);
    \path [-] (1) edge node {} (3);
    \path [-] (2) edge node {} (3);
    \path [-] (3) edge node {} (4);
    \path [-] (4) edge node {} (5);
    \path [-] (5) edge node {} (7,0);
    \path [-] (9,0) edge node {} (6);
    \path [-] (1) edge node {} (7);

    \path [-] (4) edge node {} (4l);
    \path [-] (5) edge node {} (5l);
    \path [-] (6) edge node {} (6l);
    \path [-] (7) edge node {} (7l);
\end{scope}

\begin{scope}
    \path (7,0) -- node[auto=false]{\Large \ldots} (9,0);
\end{scope}
\end{tikzpicture}}}
\qquad
\subfigure[$G(3,k+1,0)$]{
\scalebox{\c}{
\begin{tikzpicture}
\node[shape=circle,draw=black,fill=black,label=below:\Large $c_{1}$] (1) at (0,0) {};
\node[shape=circle,draw=black,fill=black,label=left:\Large $c_{2}$] (2) at (1,2) {};
\node[shape=circle,draw=black,fill=black,label=below:\Large $c_{3}$] (3) at (2,0) {};

\node[shape=circle,draw=black,fill=black,label=below:\Large $u_{1}$] (4) at (4,0) {};
\node[shape=circle,draw=black,fill=black] (4l) at (4,2) {};
\node[shape=circle,draw=black,fill=black,label=below:\Large $u_{2}$] (5) at (6,0) {};
\node[shape=circle,draw=black,fill=black] (5l) at (6,2) {};
\node[shape=circle,draw=black,fill=black,label=below:\Large $u_{k}$] (6) at (10,0) {};
\node[shape=circle,draw=black,fill=black] (6l) at (10,2) {};
\node[shape=circle,draw=black,fill=black,label=below:\Large $u_{k+1}$] (7) at (12,0) {};
\node[shape=circle,draw=black,fill=black] (7l) at (12,2) {};

\begin{scope}
    \path [-] (1) edge node {} (2);
    \path [-] (1) edge node {} (3);
    \path [-] (2) edge node {} (3);
    \path [-] (3) edge node {} (4);
    \path [-] (4) edge node {} (5);
    \path [-] (5) edge node {} (7,0);
    \path [-] (9,0) edge node {} (6);
    \path [-] (6) edge node {} (7);

    \path [-] (4) edge node {} (4l);
    \path [-] (5) edge node {} (5l);
    \path [-] (6) edge node {} (6l);
    \path [-] (7) edge node {} (7l);
\end{scope}

\begin{scope}
    \path (7,0) -- node[auto=false]{\Large \ldots} (9,0);
\end{scope}
\end{tikzpicture}}}
\caption{}%
\label{fig:BaseCase1}%
\end{figure}

\noindent Note that $G(3,k,1) - v_1 \cong $ $G(3,k+1,0) - u_{k+1} $ and hence $G(3,k,1) - v_1 \sim $ $G(3,k+1,0) - u_{k+1}$. For simplicity let $G'(3,k,1)$ and $G'(3,k+1,0)$ denote $G(3,k,1) - N[v_1]$  and $G(3,k+1,0) - N[u_{k+1}]$ respectively. Now let $u_{k}$ and $c_1$ be the vertices illustrated in Figure~\ref{fig:BaseCase2} in $G'(3,k,1)$ and $G'(3,k+1,0)$, respectively.

\setcounter{subfigure}{0}
\begin{figure}[!h]
\def\c{0.45}
\def\r{1}
\centering
\subfigure[$G'(3,k,1)$]{
\scalebox{\c}{
\begin{tikzpicture}
\node[shape=circle,draw=black,fill=black,label=left:\Large $c_{2}$] (2) at (2,2) {};
\node[shape=circle,draw=black,fill=black,label=below:\Large $c_{3}$] (3) at (2,0) {};

\node[shape=circle,draw=black,fill=black,label=below:\Large $u_{1}$] (4) at (4,0) {};
\node[shape=circle,draw=black,fill=black] (4l) at (4,2) {};
\node[shape=circle,draw=black,fill=black,label=below:\Large $u_{2}$] (5) at (6,0) {};
\node[shape=circle,draw=black,fill=black] (5l) at (6,2) {};
\node[shape=circle,draw=black,fill=black,label=below:\Large $u_{k}$] (6) at (10,0) {};
\node[shape=circle,draw=black,fill=black] (6l) at (10,2) {};

\begin{scope}
    \path [-] (2) edge node {} (3);
    \path [-] (3) edge node {} (4);
    \path [-] (4) edge node {} (5);
    \path [-] (5) edge node {} (7,0);
    \path [-] (9,0) edge node {} (6);

    \path [-] (4) edge node {} (4l);
    \path [-] (5) edge node {} (5l);
    \path [-] (6) edge node {} (6l);
\end{scope}

\begin{scope}
    \path (7,0) -- node[auto=false]{\Large \ldots} (9,0);
\end{scope}
\end{tikzpicture}}}
\qquad
\subfigure[$G'(3,k+1,0)$]{
\scalebox{\c}{
\begin{tikzpicture}
\node[shape=circle,draw=black,fill=black,label=below:\Large $c_{1}$] (1) at (0,0) {};
\node[shape=circle,draw=black,fill=black,label=left:\Large $c_{2}$] (2) at (1,2) {};
\node[shape=circle,draw=black,fill=black,label=below:\Large $c_{3}$] (3) at (2,0) {};

\node[shape=circle,draw=black,fill=black,label=below:\Large $u_{1}$] (4) at (4,0) {};
\node[shape=circle,draw=black,fill=black] (4l) at (4,2) {};
\node[shape=circle,draw=black,fill=black,label=below:\Large $u_{2}$] (5) at (6,0) {};
\node[shape=circle,draw=black,fill=black] (5l) at (6,2) {};
\node[shape=circle,draw=black,fill=black,label=below:\Large $u_{k-1}$] (6) at (10,0) {};
\node[shape=circle,draw=black,fill=black] (6l) at (10,2) {};
\node[shape=circle,draw=black,fill=black] (7l) at (12,2) {};

\begin{scope}
    \path [-] (1) edge node {} (2);
    \path [-] (1) edge node {} (3);
    \path [-] (2) edge node {} (3);
    \path [-] (3) edge node {} (4);
    \path [-] (4) edge node {} (5);
    \path [-] (5) edge node {} (7,0);
    \path [-] (9,0) edge node {} (6);

    \path [-] (4) edge node {} (4l);
    \path [-] (5) edge node {} (5l);
    \path [-] (6) edge node {} (6l);
\end{scope}

\begin{scope}
    \path (7,0) -- node[auto=false]{\Large \ldots} (9,0);
\end{scope}
\end{tikzpicture}}}
\caption{}%
\label{fig:BaseCase2}%
\end{figure}

\noindent Note that $G'(3,k,1) - u_k \cong $ $G'(3,k+1,0) - c_1 $ and $G'(3,k,1) - N[u_k] \cong $ $G'(3,k+1,0) - N[c_1]$. Therefore $G'(3,k,1) - u_k \sim $ $G'(3,k+1,0) - c_1 $ and $G'(3,k,1) - N[u_k] \sim $ $G'(3,k+1,0) - N[c_1]$. Finally by Proposition \ref{prop:deletion}, $G'(3,k,1)\sim G'(3,k+1,0)$ and hence $G(3,k,1)\sim G(3,k+1,0)$.

Now suppose $G(3,k,\ell)\sim G(3,k+\ell,0)$ for all $0 < \ell \leq m$ for some $m\ge 1$. Let $v_{m+1}$ and $u_{k+m+1}$ be the vertices illustrated in Figure~\ref{fig:InductiveStep} in $G(3,k,m+1)$ and $G(3,k+m+1,0)$, respectively.

\setcounter{subfigure}{0}
\begin{figure}[!h]
\def\c{0.45}
\def\r{1}
\centering
\subfigure[$G(3,k,m+1)$]{
\scalebox{\c}{
\begin{tikzpicture}
\node[shape=circle,draw=black,fill=black,label=below:\Large $c_{1}$] (1) at (2,0) {};
\node[shape=circle,draw=black,fill=black,label=left:\Large $c_{2}$] (2) at (3,2) {};
\node[shape=circle,draw=black,fill=black,label=below:\Large $c_{3}$] (3) at (4,0) {};

\node[shape=circle,draw=black,fill=black,label=below:\Large $u_{1}$] (4) at (6,0) {};
\node[shape=circle,draw=black,fill=black] (4l) at (6,2) {};
\node[shape=circle,draw=black,fill=black,label=below:\Large $u_{k}$] (5) at (9,0) {};
\node[shape=circle,draw=black,fill=black] (5l) at (9,2) {};

\node[shape=circle,draw=black,fill=black,label=below:\Large $v_{1}$] (6) at (0,0) {};
\node[shape=circle,draw=black,fill=black] (6l) at (0,2) {};
\node[shape=circle,draw=black,fill=black,label=below:\Large $v_{m+1}$] (7) at (-3,0) {};
\node[shape=circle,draw=black,fill=black] (7l) at (-3,2) {};

\begin{scope}
    \path [-] (1) edge node {} (2);
    \path [-] (1) edge node {} (3);
    \path [-] (2) edge node {} (3);
    
    \path [-] (3) edge node {} (4);
    \path [-] (4) edge node {} (6.5,0);
    \path [-] (8.5,0) edge node {} (5);
    \path [-] (4) edge node {} (4l);
    \path [-] (5) edge node {} (5l);

    \path [-] (1) edge node {} (6);
    \path [-] (6) edge node {} (-0.5,0);
    \path [-] (-2.5,0) edge node {} (7);
    \path [-] (6) edge node {} (6l);
    \path [-] (7) edge node {} (7l);
\end{scope}

\begin{scope}
    \path (6.5,0) -- node[auto=false]{\Large \ldots} (8.5,0);
    \path (-0.5,0) -- node[auto=false]{\Large \ldots} (-2.5,0);
\end{scope}
\end{tikzpicture}}}
\qquad
\subfigure[$G(3,k+m+1,0)$]{
\scalebox{\c}{
\begin{tikzpicture}
\node[shape=circle,draw=black,fill=black,label=below:\Large $c_{1}$] (1) at (0,0) {};
\node[shape=circle,draw=black,fill=black,label=left:\Large $c_{2}$] (2) at (1,2) {};
\node[shape=circle,draw=black,fill=black,label=below:\Large $c_{3}$] (3) at (2,0) {};

\node[shape=circle,draw=black,fill=black,label=below:\Large $u_{1}$] (4) at (4,0) {};
\node[shape=circle,draw=black,fill=black] (4l) at (4,2) {};
\node[shape=circle,draw=black,fill=black,label=below:\Large $u_{2}$] (5) at (6,0) {};
\node[shape=circle,draw=black,fill=black] (5l) at (6,2) {};
\node[shape=circle,draw=black,fill=black,label=below:\Large $u_{k+m+1}$] (6) at (10,0) {};
\node[shape=circle,draw=black,fill=black] (6l) at (10,2) {};

\begin{scope}
    \path [-] (1) edge node {} (2);
    \path [-] (1) edge node {} (3);
    \path [-] (2) edge node {} (3);
    \path [-] (3) edge node {} (4);
    \path [-] (4) edge node {} (5);
    \path [-] (5) edge node {} (7,0);
    \path [-] (9,0) edge node {} (6);

    \path [-] (4) edge node {} (4l);
    \path [-] (5) edge node {} (5l);
    \path [-] (6) edge node {} (6l);
\end{scope}

\begin{scope}
    \path (7,0) -- node[auto=false]{\Large \ldots} (9,0);
\end{scope}
\end{tikzpicture}}}
\caption{}%
\label{fig:InductiveStep}%
\end{figure}

\noindent Note that $G(3,k,m+1) - v_{m+1} \cong G(3,k,m) \cup K_1$ and $G(3,k+m+1,0) - u_{k+m+1} \cong G(3,k+m,0) \cup K_1$. By our inductive hypothesis and Proposition \ref{prop:deletion} $G(3,k,m+1) - v_{m+1} \sim G(3,k+m+1,0) - u_{k+m+1}$. Furthermore $G(3,k,m+1) - N[v_{m+1}] \cong G(3,k,m-1) \cup K_1$ and $G(3,k+m+1,0) - N[u_{k+m+1}] \cong G(3,k+m-1,0) \cup K_1$. Again by our inductive hypothesis and Proposition \ref{prop:deletion}, $G(3,k,m+1) - N[v_{m+1}] \sim G(3,k+m+1,0) - N[u_{k+m+1}]$. Therefore by Proposition \ref{prop:graphstarsequal}, $G(3,k,m+1) \sim G(3,k+m+1,0)$ and by induction $G(3,k+\ell,0) \sim G(3,k,\ell)$ for all $\ell > 0$. Furthermore $G(3,k+\ell,0) \sim G(3,k,\ell)$ for all $\ell,k > 0$.

It now suffices to show $G(5,k,0) \sim G(3,k+1,0)$ and $G(5,k,\ell) \sim G(3,k+1,\ell)$ for all $\ell,k > 0$. We will begin with showing $G(5,k,0) \sim G(3,k+1,0)$. Let $c_{3}$ and $u_1$ be the vertices illustrated in Figure~\ref{fig:5Case0} in $G(5,k,0)$ and $G(3,k+1,0)$, respectively.

\setcounter{subfigure}{0}
\begin{figure}[!h]
\def\c{0.45}
\def\r{1}
\centering
\subfigure[$G(5,k,0)$]{
\scalebox{\c}{
\begin{tikzpicture}
\node[shape=circle,draw=black,fill=black,label=below:\Large $c_{1}$] (c1) at (-1,0) {};
\node[shape=circle,draw=black,fill=black,label=below:\Large $c_{2}$] (c2) at (0.5,-0.5) {};
\node[shape=circle,draw=black,fill=black,label=below:\Large $c_{3}$] (c3) at (2,0) {};
\node[shape=circle,draw=black,fill=black,label=right:\Large $c_{4}$] (c4) at (1.5,2) {};
\node[shape=circle,draw=black,fill=black,label=left:\Large $c_{5}$] (c5) at (-0.5,2) {};

\node[shape=circle,draw=black,fill=black,label=below:\Large $u_{1}$] (4) at (4,0) {};
\node[shape=circle,draw=black,fill=black] (4l) at (4,2) {};
\node[shape=circle,draw=black,fill=black,label=below:\Large $u_{2}$] (5) at (6,0) {};
\node[shape=circle,draw=black,fill=black] (5l) at (6,2) {};
\node[shape=circle,draw=black,fill=black,label=below:\Large $u_{k}$] (6) at (10,0) {};
\node[shape=circle,draw=black,fill=black] (6l) at (10,2) {};

\begin{scope}
    \path [-] (c1) edge node {} (c2);
    \path [-] (c2) edge node {} (c3);
    \path [-] (c3) edge node {} (c4);
    \path [-] (c4) edge node {} (c5);
    \path [-] (c5) edge node {} (c1);
    \path [-] (c3) edge node {} (4);
    \path [-] (4) edge node {} (5);
    \path [-] (5) edge node {} (7,0);
    \path [-] (9,0) edge node {} (6);

    \path [-] (4) edge node {} (4l);
    \path [-] (5) edge node {} (5l);
    \path [-] (6) edge node {} (6l);
\end{scope}

\begin{scope}
    \path (7,0) -- node[auto=false]{\Large \ldots} (9,0);
\end{scope}
\end{tikzpicture}}}
\qquad
\subfigure[$G(3,k+1,0)$]{
\scalebox{\c}{
\begin{tikzpicture}
\node[shape=circle,draw=black,fill=black,label=below:\Large $c_{1}$] (1) at (0,0) {};
\node[shape=circle,draw=black,fill=black,label=left:\Large $c_{2}$] (2) at (1,2) {};
\node[shape=circle,draw=black,fill=black,label=below:\Large $c_{3}$] (3) at (2,0) {};

\node[shape=circle,draw=black,fill=black,label=below:\Large $u_{1}$] (4) at (4,0) {};
\node[shape=circle,draw=black,fill=black] (4l) at (4,2) {};
\node[shape=circle,draw=black,fill=black,label=below:\Large $u_{2}$] (5) at (6,0) {};
\node[shape=circle,draw=black,fill=black] (5l) at (6,2) {};
\node[shape=circle,draw=black,fill=black,label=below:\Large $u_{k+1}$] (6) at (10,0) {};
\node[shape=circle,draw=black,fill=black] (6l) at (10,2) {};

\begin{scope}
    \path [-] (1) edge node {} (2);
    \path [-] (1) edge node {} (3);
    \path [-] (2) edge node {} (3);
    \path [-] (3) edge node {} (4);
    \path [-] (4) edge node {} (5);
    \path [-] (5) edge node {} (7,0);
    \path [-] (9,0) edge node {} (6);

    \path [-] (4) edge node {} (4l);
    \path [-] (5) edge node {} (5l);
    \path [-] (6) edge node {} (6l);
\end{scope}

\begin{scope}
    \path (7,0) -- node[auto=false]{\Large \ldots} (9,0);
\end{scope}
\end{tikzpicture}}}
\caption{}%
\label{fig:5Case0}%
\end{figure}

Note that $G(5,k,0) - c_3 \cong P_k^* \cup P_4$ and $G(3,k+1,0) - u_1 \cong P_k^* \cup C_3 \cup K_1$. Furthermore $I(P_4,x)=I(C_3 \cup K_1,x)=1+4x+3x^2$. Thus by Proposition \ref{prop:deletion} $G(5,k,0) - c_3 \sim G(3,k+1,0) - u_1$. As $G(5,k,0) - N[c_3]$ and $G(3,k+1,0) - N[u_1]$ are both isomorphic to $P_{k-1}^* \cup K_2 \cup K-1$ then $G(5,k,0) - N[c_3] \sim G(3,k+1,0) - N[u_1]$. Therefore by Proposition \ref{prop:deletion}, $G(5,k,0) \sim G(3,k+1,0)$.

Finally we will show $G(5,k,\ell) \sim G(3,k+1,\ell)$. Consider the vertex $v_1$ illustrated in Figure~\ref{fig:5Casek} for both $G(5,k,\ell)$ and $G(3,k+1,\ell)$.

\setcounter{subfigure}{0}
\begin{figure}[!h]
\def\c{0.45}
\def\r{1}
\centering
\subfigure[$G(5,k,\ell)$]{
\scalebox{\c}{
\begin{tikzpicture}

\node[shape=circle,draw=black,fill=black,label=below:\Large $c_{1}$] (c1) at (1.5,0) {};
\node[shape=circle,draw=black,fill=black,label=below:\Large $c_{2}$] (c2) at (3,-0.5) {};
\node[shape=circle,draw=black,fill=black,label=below:\Large $c_{3}$] (c3) at (4.5,0) {};
\node[shape=circle,draw=black,fill=black,label=right:\Large $c_{4}$] (c4) at (4,2) {};
\node[shape=circle,draw=black,fill=black,label=left:\Large $c_{5}$] (c5) at (2,2) {};

\node[shape=circle,draw=black,fill=black,label=below:\Large $u_{1}$] (4) at (6,0) {};
\node[shape=circle,draw=black,fill=black] (4l) at (6,2) {};
\node[shape=circle,draw=black,fill=black,label=below:\Large $u_{k}$] (5) at (9,0) {};
\node[shape=circle,draw=black,fill=black] (5l) at (9,2) {};

\node[shape=circle,draw=black,fill=black,label=below:\Large $v_{1}$] (6) at (0,0) {};
\node[shape=circle,draw=black,fill=black] (6l) at (0,2) {};
\node[shape=circle,draw=black,fill=black,label=below:\Large $v_{\ell}$] (7) at (-3,0) {};
\node[shape=circle,draw=black,fill=black] (7l) at (-3,2) {};

\begin{scope}

    \path [-] (c1) edge node {} (c2);
    \path [-] (c2) edge node {} (c3);
    \path [-] (c3) edge node {} (c4);
    \path [-] (c4) edge node {} (c5);
    \path [-] (c5) edge node {} (c1);
    
    \path [-] (c3) edge node {} (4);
    \path [-] (4) edge node {} (6.5,0);
    \path [-] (8.5,0) edge node {} (5);
    \path [-] (4) edge node {} (4l);
    \path [-] (5) edge node {} (5l);

    \path [-] (c1) edge node {} (6);
    \path [-] (6) edge node {} (-0.5,0);
    \path [-] (-2.5,0) edge node {} (7);
    \path [-] (6) edge node {} (6l);
    \path [-] (7) edge node {} (7l);
\end{scope}

\begin{scope}
    \path (6.5,0) -- node[auto=false]{\Large \ldots} (8.5,0);
    \path (-0.5,0) -- node[auto=false]{\Large \ldots} (-2.5,0);
\end{scope}
\end{tikzpicture}}}
\qquad
\subfigure[$G(3,k+1,\ell)$]{
\scalebox{\c}{
\begin{tikzpicture}
\node[shape=circle,draw=black,fill=black,label=below:\Large $c_{1}$] (1) at (1,0) {};
\node[shape=circle,draw=black,fill=black,label=left:\Large $c_{2}$] (2) at (2,2) {};
\node[shape=circle,draw=black,fill=black,label=below:\Large $c_{3}$] (3) at (3,0) {};

\node[shape=circle,draw=black,fill=black,label=below:\Large $u_{1}$] (4) at (4,0) {};
\node[shape=circle,draw=black,fill=black] (4l) at (4,2) {};
\node[shape=circle,draw=black,fill=black,label=below:\Large $u_{k}$] (5) at (7,0) {};
\node[shape=circle,draw=black,fill=black] (5l) at (7,2) {};
\node[shape=circle,draw=black,fill=black,label=below:\Large $u_{k+1}$] (8) at (9,0) {};
\node[shape=circle,draw=black,fill=black] (8l) at (9,2) {};

\node[shape=circle,draw=black,fill=black,label=below:\Large $v_{1}$] (6) at (0,0) {};
\node[shape=circle,draw=black,fill=black] (6l) at (0,2) {};
\node[shape=circle,draw=black,fill=black,label=below:\Large $v_{\ell}$] (7) at (-3,0) {};
\node[shape=circle,draw=black,fill=black] (7l) at (-3,2) {};

\begin{scope}
    \path [-] (1) edge node {} (2);
    \path [-] (1) edge node {} (3);
    \path [-] (2) edge node {} (3);
    
    \path [-] (3) edge node {} (4);
    \path [-] (4) edge node {} (4.5,0);
    \path [-] (6.5,0) edge node {} (5);
    \path [-] (4) edge node {} (4l);
    \path [-] (5) edge node {} (5l);
    \path [-] (5) edge node {} (8);
    \path [-] (8) edge node {} (8l);

    \path [-] (1) edge node {} (6);
    \path [-] (6) edge node {} (-0.5,0);
    \path [-] (-2.5,0) edge node {} (7);
    \path [-] (6) edge node {} (6l);
    \path [-] (7) edge node {} (7l);
\end{scope}

\begin{scope}
    \path (4.5,0) -- node[auto=false]{\Large \ldots} (6.5,0);
    \path (-0.5,0) -- node[auto=false]{\Large \ldots} (-2.5,0);
\end{scope}
\end{tikzpicture}}}
\caption{}%
\label{fig:5Casek}%
\end{figure}

Note that $G(5,k,\ell) - v_1 \cong G(5,k,0) \cup P_{\ell-1}^* \cup K_1$ and $G(3,k+1,\ell) - v_1 \cong G(3,k+1,0) \cup P_{\ell-1}^* \cup K_1$. As $G(5,k,0) \sim G(3,k+1,0)$ then $G(5,k,\ell) - v_1 \sim G(3,k+1,\ell) - v_1$. Furthermore $G(5,k,\ell) - N[v_1]$ and $G(3,k+1,\ell) - N[v_1]$ and both isomorphic to $P_{\ell-2}^* \cup P_{k+1}^* \cup K_1$. Therefore $G(5,k,\ell) - N[v_1] \sim G(3,k+1,\ell) - N[v_1]$ and by Proposition \ref{prop:deletion} $G(5,k,\ell) \sim G(3,k+1,\ell)$.
\end{proof}

\begin{theorem}
\label{thm:unioddmin}
Let $n\ge 3$ be odd, $G$ be a connected well-covered unicyclic graph of order $n$ and $H_n$ a graph of order $n$ in $\mathcal{S}_P$. Then 

$$ G\succeq \left\{
\begin{array}{ll}
      C_n &~~~~\textnormal{if } n\le 7 \\      
      & \\
      H_n &~~~~\textnormal{if } n\ge 9 \\           
\end{array}\right.$$
\end{theorem}
\begin{proof}
Let $G$ be a connected well-covered unicyclic graph of order $n$ where $n$ is odd. If $n\le 7$, then from Theorem~\ref{thm:wcunicyclicgraphs}, $C_n$ is a well-covered unicyclic graph. So by Theorem~\ref{thm:maxminunicyclic}, $C_n\preceq G$. For $n\ge 9$, we know that $G\in \mathcal{S}_3\cup \mathcal{S}_5$. Let $H'$ be the graph obtained from $G$ by replacing each well-covered branch $S$ with a well-covered branch isomorphic to $P_{k-1}^{\ast}$ with exactly one vertex of degree $2$ (as in Lemma~\ref{lem:replacebranchwithpath}). By Lemma \ref{lem:replacebranchwithpath}, $H' \preceq G$. Furthermore clearly $H' \in \mathcal{S}_P$. Thus by Lemma~\ref{lem:equivofoddwcuniminimal}, $I(H',x)=I(H_n,x)$ and thus $H_n\preceq G$.

\end{proof}

\begin{corollary}
Let $G$ be a connected well-covered unicyclic graph of odd order $n$ and $H_n\in \mathcal{S}_P$ of order $n$. Then

\begin{itemize}
\item[i)] $\xi(C_n)\le \xi(G)\le\xi(M_n)$ if $n\le 7$, and
\item[ii)] $\xi(H_n)\le \xi(G)\le\xi(M_n)$ if $n\ge 9$.
\end{itemize} 
\end{corollary}

\section{Resolutions to open questions and conjectures related to $\preceq$}\label{sec:counterexamples}
We have seen in the previous sections that when $G$ is a (well-covered) tree or a (well-covered) unicyclic graph there seems to be a correspondence between graphs that maximize/minimize the maximum degree and the graphs that are maximal/minimal with respect to $\preceq$.  This lead Oboudi to ask two questions about the degree sequence of trees and their independence polynomials. Let $D_G$ denote the degree sequence of $G$ in nondecreasing order. $D_H=(y_1,y_2,\ldots, y_n) \preceq D_G=(x_1,x_2,\ldots ,x_n)$ if $x_t>y_t$ for some $t\in \{1,2,\ldots , n\}$ and $x_k=y_k$ for all $k<t$.

\begin{ques}[\cite{Oboudi2018treeroots}]\label{ques:oboudi1}
Let $T_1$ and $T_2$ be two trees such that $D_{T_1}\prec D_{T_2}$. Is it true that $T_1\prec T_2$?
\end{ques}

\begin{ques}[\cite{Oboudi2018treeroots}]\label{ques:oboudi2}
Let $T_1$ and $T_2$ be two trees such that $I(T_1,x)=I(T_2,x)$. Is it true that $D_{T_1}=D_{T_2}$?
\end{ques}

These questions are closely related to the $15$-year old conjecture due to Levit and Mandrescu. 

\begin{conj}[\cite{INDPOLY,Levit2008}]
\label{conj:wctree}
If $G$ is a connected graph and $T$ is a well-covered tree, and $I(T,x)=I(G,x)$, then $G$ is a well-covered tree.
\end{conj}

In particular, if Question~\ref{ques:oboudi2} has a positive answer, then any connected graph independence equivalent to a well-covered tree of order $2n$ must have exactly $n$ leaves and therefore would provide some evidence for Conjecture~\ref{conj:wctree}. However, as we will show, Conjecture~\ref{conj:wctree} is false and the answer to both Questions~\ref{ques:oboudi1} and \ref{ques:oboudi2} is ``no''. We now provide an infinite family of graphs that will justify all three claims in the previous sentence. 


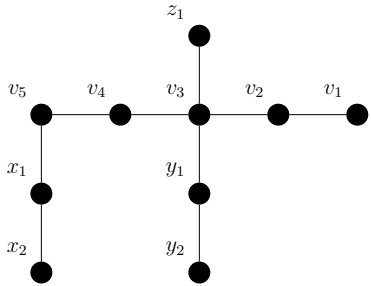
\begin{figure}[htp]
\def\c{0.7}
\def\r{1.5}
\centering
\scalebox{\c}{
\begin{tikzpicture}
\begin{scope}[every node/.style={circle,thick,draw,fill}]
    \node[label=above left:$v_5$] (11) at (1*\r,0*\r) {};
    \node[label=above left:$x_1$] (2) at (1*\r,-1*\r) {};
    \node[label=above left:$x_2$] (8) at (1*\r,-2*\r) {};
    \node[label=above left:$v_4$] (5) at (2*\r,0*\r) {};
    \node[label=above left:$v_3$] (12) at (3*\r,0*\r) {};
    \node[label=above left:$z_1$] (6) at (3*\r,1*\r) {};
    \node[label=above left:$y_1$] (3) at (3*\r,-1*\r) {};
    \node[label=above left:$y_2$] (9) at (3*\r,-2*\r) {}; 
    \node[label=above left:$v_2$] (4) at (4*\r,0*\r) {}; 
    \node[label=above left:$v_1$] (10) at (5*\r,0*\r) {};    
\end{scope}

\begin{scope}
    \path [-] (11) edge node {} (2);
    \path [-] (2) edge node {} (8);
    
    \path [-] (11) edge node {} (5);
    \path [-] (5) edge node {} (12);
    \path [-] (12) edge node {} (6);
    \path [-] (12) edge node {} (4);
    \path [-] (4) edge node {} (10);
    \path [-] (12) edge node {} (3);
    \path [-] (3) edge node {} (9);
    
\end{scope}
\end{tikzpicture}}
\caption{The tree $G_{10}$ that is not well-covered but is equivalent to $P_{5}^{\ast}$}%
\label{fig:T1}%
\end{figure}

The graph in Figure~\ref{fig:T1} has independence polynomial $13x^5+45x^4+59x^3+36x^2+10x+1$ which is equivalent to the independence polynomial of $P_{5}^{\ast}$. From computations with the computer algebra system Maple, this is the smallest counterexample to Conjecture~\ref{conj:wctree}. It is a counterexample to Question \ref{ques:oboudi2} since $\Delta(P_{5}^{\ast})=3$ but $\Delta(G_{10})=4$ so $D_{P_{5}^{\ast}}\neq D_{G_{10}}$, but the graphs are independence equivalent. Furthermore is a counterexample to Question \ref{ques:oboudi1} since $G_{10}$ has one fewer leaf than $P_{5}^{\ast}$ so $D_{P_{5}^{\ast}} \prec D_{G_{10}}$, but again the graphs are independence equivalent so $G_{10} \preceq P_{5}^{\ast}$. This graph also provides the basis for an infinite family of counterexamples for each of Question \ref{ques:oboudi1}, Question \ref{ques:oboudi2}, and Conjecture~\ref{conj:wctree}.

We now define the graph $G_{2n}$ recursively as follows: Let $G_{10}$ be the graph pictured in Figure~\ref{fig:T1}. The graph $G_{2(n+1)}$ is obtained from $G_{2n}$ by adding a copy of $K_2$, with vertices labelled $v_{n+1}$ and $w_{n+1}$, and joining $v_{n}$ and $v_{n+1}$ with an edge, see Figure~\ref{fig:G2n}.

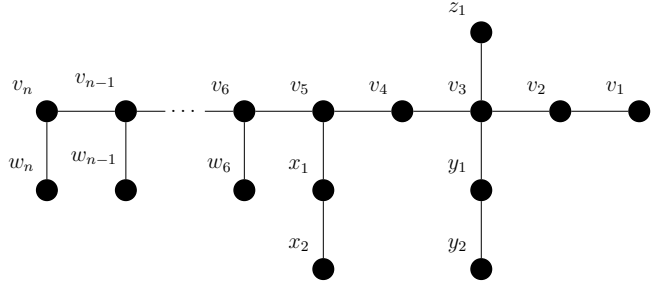
\begin{figure}[htp]
\def\c{0.7}
\def\r{1.5}
\centering
\scalebox{\c}{
\begin{tikzpicture}

\begin{scope}[every node/.style={circle,thick,draw,fill}]
     \node[label=above left:$w_n$] (13) at (-2.5*\r,-1*\r) {};
    \node[label=above left:$v_n$] (14) at (-2.5*\r,0*\r) {};
    
     \node[label=above left:$w_6$] (15) at (0*\r,-1*\r) {};
    \node[label=above left:$v_6$] (16) at (0*\r,0*\r) {};
    
    \node[label=above left:$w_{n-1}$] (7) at (-1.5*\r,-1*\r) {};
    \node[label=above left:$v_{n-1}$] (1) at (-1.5*\r,0*\r) {};
    \node[label=above left:$v_{5}$] (11) at (1*\r,0*\r) {};
    \node[label=above left:$x_1$] (2) at (1*\r,-1*\r) {};
    \node[label=above left:$x_2$] (8) at (1*\r,-2*\r) {};
    \node[label=above left:$v_4$] (5) at (2*\r,0*\r) {};
    \node[label=above left:$v_3$] (12) at (3*\r,0*\r) {};
    \node[label=above left:$z_1$] (6) at (3*\r,1*\r) {};
    \node[label=above left:$y_1$] (3) at (3*\r,-1*\r) {};
    \node[label=above left:$y_2$] (9) at (3*\r,-2*\r) {}; 
    \node[label=above left:$v_2$] (4) at (4*\r,0*\r) {}; 
    \node[label=above left:$v_1$] (10) at (5*\r,0*\r) {};    
\end{scope}

\begin{scope}
    
    \path [-] (7) edge node {} (1);
    \path [-] (11) edge node {} (2);
    \path [-] (2) edge node {} (8);
    \path [-] (14) edge node {} (1);
    \path [-] (14) edge node {} (13);
    \path [-] (11) edge node {} (5);
    \path [-] (5) edge node {} (12);
    \path [-] (12) edge node {} (6);
    \path [-] (15) edge node {} (16);
    \path [-] (11) edge node {} (16);
    \path [-] (12) edge node {} (4);
    \path [-] (4) edge node {} (10);
    \path [-] (12) edge node {} (3);
    \path [-] (3) edge node {} (9);
    
\end{scope}
	
	\path (1) -- node[auto=false]{\ldots} (16);
	\path [-] (16) edge node {} (-0.5*\r,0*\r) ;
    \path [-] (-1*\r,0*\r) edge node {} (1);

\end{tikzpicture}}
\caption{The graph $G_{2n}$ that is equivalent to $P_{n}^{\ast}$}%
\label{fig:G2n}%
\end{figure}

\begin{theorem}\label{thm:indequiv2}
For $n\ge 6$, $I(P_n^{\ast},x)=I(G_{2n},x)$.
\end{theorem}
\begin{proof}
The proof is by induction on $n$. For $n=6$ and $n=7$ we have equivalence from direct computations. Now suppose $I(P_k^{\ast},x)=I(G_{2k},x)$ for all $7\le k< n$. Label the vertices of $P_n^{\ast}$ such that the vertices of $P_n$ are labelled with $u_1,u_2,\ldots,u_n$ such that $u_{i}\sim u_{i+1}$ for $i=1,2,\ldots n-1$ and label the leaf adjacent to $u_{i}$ with $u_{i}^{\ast}$. Label the vertices of $G_{2n}$ as in its definition (see, Figure~\ref{fig:G2n}). Now, by Proposition~\ref{prop:deletion},
\begin{align*}
I(P_n^{\ast},x)&=I(P_n^{\ast}-u_1,x)+xI(P_n^{\ast}-N[u_1],x)\\
&=(1+x)I(P_{n-1}^{\ast},x)+x(1+x)I(P_{n-1}^{\ast},x)\\
&=(1+x)I(G_{2(n-1)},x)+x(1+x)I(G_{2(n-2)},x)\\ 
&=I(G_{2(n-1)}\cup K_1,x)+xI(G_{2(n-2)}\cup K_1,x)\\
&= I(G_{2n}-v_n,x)+xI(G_{2n}-N[v_n],x)\\
&=I(G_{2n},x).
\end{align*}

The result follows by induction.

\end{proof}

Since $v_4\in G_{2n}$ is not adjacent to a leaf, $G_{2n}$ is not well-covered by Corollary~\ref{cor:wctreesaregraphstars}. Thus, Theorem~\ref{thm:indequiv2} provides an infinite family of counterexamples to Conjecture \ref{conj:wctree}. This is also an infinite family answering Questions~\ref{ques:oboudi1} and \ref{ques:oboudi2} in the negative since $G_{2n}$ has one fewer pendant vertex as $P_{n}^{\ast}$ so $D_{P_{5}^{\ast}} \prec D_{G_{10}}$ ($D_{P_{5}^{\ast}} \neq D_{G_{10}}$ ) but the graphs are independence equivalent so $G_{2n} \preceq P_{n}^{\ast}$.


There is another conjecture related to $\preceq$ posed by Oboudi that we will answer now as well.

\begin{conj}[\cite{Oboudi2018treeroots}]\label{conj:oboudi3}
Let $T_1$ and $T_2$ be two trees of order $n$. Is it true that $T_1\prec T_2$ or $T_2\preceq T_1$?
\end{conj}

Although Oboudi noted that all trees of order at most $6$ are comparable and we have verified that all trees of order $7$ are also all comparable, the trend stops at order $8$. The trees $T_1$ and $T_2$ in Figure~\ref{fig:incomptrees} both have order $8$, with $\xi(T_1)\approx -0.2451223338$ and $\xi(T_2)\approx -0.2410859067$, but $I(T_2,-0.1)>I(T_1,-0.1)$ and $I(T_2,-0.2)<I(T_1,-0.2)$. Hence $T_1$ and $T_2$ are incomparable with respect to $\preceq$.
\setcounter{subfigure}{0}
\begin{figure}[htb]
\def\c{0.5}
\def\r{2}
\centering
\subfigure[$T_1$]{
\scalebox{\c}{
\begin{tikzpicture}
\begin{scope}
	\node[shape=circle,draw=black,fill=black] (1) at (0*\r,0*\r) {};
	\node[shape=circle,draw=black,fill=black] (2) at (1*\r,0*\r) {};
	\node[shape=circle,draw=black,fill=black] (3) at (2*\r,-1*\r) {};
	\node[shape=circle,draw=black,fill=black] (4) at (2*\r,0*\r) {};
	\node[shape=circle,draw=black,fill=black] (5) at (2*\r,1*\r) {};
	\node[shape=circle,draw=black,fill=black] (6) at (-1*\r,0*\r) {};
	\node[shape=circle,draw=black,fill=black] (7) at (-2*\r,-1*\r) {};
	\node[shape=circle,draw=black,fill=black] (8) at (-2*\r,1*\r) {};
\end{scope}
\begin{scope}
    \path [-] (1) edge node {} (2);
    \path [-] (1) edge node {} (6);
    
    \path [-] (3) edge node {} (2);
    \path [-] (4) edge node {} (2);
    \path [-] (5) edge node {} (2);
    
    \path [-] (6) edge node {} (7);
    \path [-] (6) edge node {} (8);
  
\end{scope}

\end{tikzpicture}}}
\qquad
\subfigure[$T_2$]{
\scalebox{\c}{
\begin{tikzpicture}
\begin{scope}
	\node[shape=circle,draw=black,fill=black] (1) at (0*\r,0*\r) {};
	\node[shape=circle,draw=black,fill=black] (2) at (1*\r,0*\r) {};
	\node[shape=circle,draw=black,fill=black] (3) at (2*\r,0*\r) {};
	\node[shape=circle,draw=black,fill=black] (4) at (0*\r,1*\r) {};
	\node[shape=circle,draw=black,fill=black] (5) at (0*\r,-1*\r) {};
	\node[shape=circle,draw=black,fill=black] (6) at (-1*\r,0*\r) {};
	\node[shape=circle,draw=black,fill=black] (7) at (0*\r,-2*\r) {};
	\node[shape=circle,draw=black,fill=black] (8) at (-2*\r,0*\r) {};
\end{scope}
\begin{scope}
    \path [-] (1) edge node {} (2);
    \path [-] (1) edge node {} (6);
    \path [-] (1) edge node {} (4);
    \path [-] (1) edge node {} (5);
    
    \path [-] (3) edge node {} (2);

    \path [-] (5) edge node {} (7);

    \path [-] (6) edge node {} (8);
  
\end{scope}

\end{tikzpicture}}}
\caption{$T_1$ and $T_2$ are incomparable trees of order $8$ with respect to $\preceq$.}%
\label{fig:incomptrees}%
\end{figure}
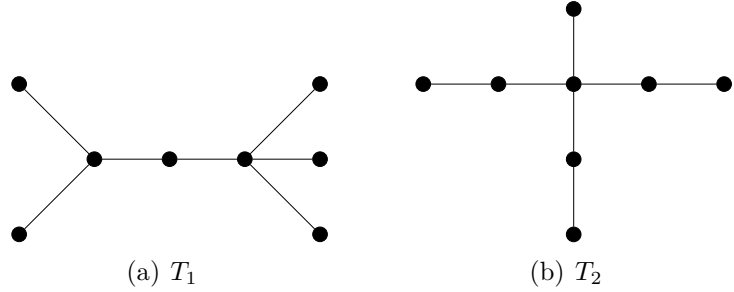

Note that by Lemma \ref{lem:relationresptostar}, $T_1^{\ast}$ and $T_2^{\ast}$ are also incomparable with respect to $\preceq$. In fact, applying the graph star operation an equal number of times to both $T_1$ and $T_2$ will always produce two incomparable trees with respect to $\preceq$. This gives the following theorem.

\begin{theorem}
If $k\ge 0$ and $n=2^{k+3}$, then there exist two trees of order $n$ that are incomparable with respect to $\preceq$. 
\end{theorem}

\section{Conclusion}
We conclude by asking some open questions. In \cite{Csikvari2013posetII}, Csikv\'{a}ri showed that the partial order induced by $\preceq$ over all trees of fixed order had a unique maximum graph, $S_{n}$, and a unique minimum graph, $P_{n}$. We extended these results to connected unicyclic graphs as well as well-covered trees and connected well-covered unicyclic graphs. In the case of connected unicyclic graphs of fixed order, there were two minimum graphs, $C_n$ and $D_n$. However $C_n \sim D_n$ therefore we may consider the independence equivalence class $[C_n]$ as the minimum element. Moreover, in the each of the other cases there was a unique maximum and unique minimum independence equivalence class. This leads us to our first question.

\begin{ques}
Are there any interesting families of graphs $\mathcal{F}$, such that there are no maximum or minimum independence equivalence classes with respect to $\preceq$?
\end{ques}

Although we have shown that $\preceq$ is not a total order on trees of equal order, open problems still remain about the properties of this partial order.

\begin{ques}
Are there any interesting families of graphs $\mathcal{F}$, such that $\preceq$ is a total order?
\end{ques}

\begin{ques}
Can there be arbitrarily large anti-chains with respect to $\preceq$ among trees of equal order?
\end{ques}

What about the maximum and minimum graphs of other families of graphs? There are some that generalize trees that follow relatively easily form our work and Csikv\'{a}ri's \cite{Csikvari2013posetII} work. For example, connected bipartite graphs.
\begin{theorem}
If $G$ is a connected bipartite graph of order $n$, then $$P_n\preceq G\preceq K_{\left\lceil\frac{n}{2}\right\rceil,\left\lfloor\frac{n}{2}\right\rfloor}.$$
\end{theorem}
\begin{proof}
Let $G$ be a bipartite graph of order $n$ with bipartition $(X,Y)$ where $|X|=\left\lceil\frac{n}{2}\right\rceil+k$ and $|Y|=\left\lfloor\frac{n}{2}\right\rfloor-k$ for some integer $0\le k\le\left\lfloor\frac{n}{2}\right\rfloor-1 $. Therefore, $G$ is a subgraph of $K_{|X|,|Y|}$ and by Theorem~\ref{thm:subgraphlessthangraph}, $G\preceq K_{|X|,|Y|}$. Therefore, it suffices to show that $K_{|X|,|Y|}\preceq K_{\left\lceil\frac{n}{2}\right\rceil,\left\lfloor\frac{n}{2}\right\rfloor}$. For simplicity let $f_{|X|}=I(K_{|X|,|Y|},x)$ and $f_{\left\lceil\frac{n}{2}\right\rceil}=I(K_{\left\lceil\frac{n}{2}\right\rceil,\left\lfloor\frac{n}{2}\right\rfloor},x)$. Now for all $x\in [\xi(K_{\left\lceil\frac{n}{2}\right\rceil,\left\lfloor\frac{n}{2}\right\rfloor}),0]$,
\begin{align*}
f_{|X|}-f_{\left\lceil\frac{n}{2}\right\rceil}&=(1+x)^{|X|}+(1+x)^{|Y|}-1-\left( (1+x)^{\left\lceil\frac{n}{2}\right\rceil}+(1+x)^{\left\lfloor\frac{n}{2}\right\rfloor-1}\right)\\
&=(1+x)^{|Y|}\left( (1+x)^{|X|-|Y|}+1-\left( (1+x)^{|X|-|Y|-k}+(1+x)^{k} \right)\right)\\
&=(1+x)^{|Y|}\left(I(\overline{K_{|X|-|Y|}},x)-I(K_{|X|-|Y|-k,k},x) \right). \tag{*}\label{eq:bipartite}
\end{align*}
Now, $(1+x)^{|Y|}\ge 0$ for all $x$ in the desired interval, and $\overline{K_{|X|-|Y|}}$ is a subgraph of $K_{|X|-|Y|-k,k}$ so $I(\overline{K_{|X|-|Y|}},x)-I(K_{|X|-|Y|-k,k},x)\ge 0$ for all $x\in [\xi(K_{|X|-|Y|-k,k}),0]$ by Theorem~\ref{thm:subgraphlessthangraph}. Also by Theorem~\ref{thm:subgraphlessthangraph}, $\xi(K_{|X|-|Y|-k,k})\le \xi(K_{\left\lceil\frac{n}{2}\right\rceil,\left\lfloor\frac{n}{2}\right\rfloor})$. Therefore, (\ref{eq:bipartite}) is nonnegative, so $G\preceq K_{\left\lceil\frac{n}{2}\right\rceil,\left\lfloor\frac{n}{2}\right\rfloor}$.

Finally, the fact that $P_n\preceq G$ follows from the fact that $G$ is connected and Theorem~\ref{thm:subgraphlessthangraph} and Theorem~\ref{thm:treebounds}.

\end{proof}


A family of graphs that generalizes bipartite graphs (and therefore trees), is the family of triangle free graphs. While there are many more triangle free graphs than bipartite graphs, computations lead us to conjecture to following.

\begin{conj}
If $G$ is a connected triangle free graph of order $n$, then
$$P_n\preceq G\preceq K_{\left\lceil\frac{n}{2}\right\rceil,\left\lfloor\frac{n}{2}\right\rfloor}.$$
\end{conj}

\bibliographystyle{abbrv}
\bibliography{Unicyclic}

\end{document}